\documentclass[a4paper, reqno, oneside]{amsart}

\usepackage[latin1]{inputenc}
\usepackage{amsmath}
\usepackage{amsfonts}
\usepackage{amssymb}
\usepackage{graphicx}
\usepackage{enumerate}
\usepackage[arrow, matrix, curve]{xy}
\usepackage{color}
\usepackage{mathrsfs}
\usepackage{hyperref}

\newtheorem{theorem}{Theorem}[section]

\theoremstyle{plain}

\newtheorem{case}{Case}

\numberwithin{subcase}{case}

\newtheorem{claim}{Claim}

\newtheorem{corollary}[theorem]{Corollary}

\newtheorem{lemma}[theorem]{Lemma}

\newtheorem{proposition}[theorem]{Proposition}

\numberwithin{equation}{section}

\long\def\symbolfootnote[#1]#2{\begingroup%
\def\thefootnote{\fnsymbol{footnote}}\footnote[#1]{#2}\endgroup}

\begin{document}

\author{Karl Heuer}

\symbolfootnote[0]{\textcopyright 2019. This manuscript version is made available under the CC-BY-NC-ND 4.0 license \url{http://creativecommons.org/licenses/by-nc-nd/4.0/}}

\title[]{A sufficient condition for Hamiltonicity in locally finite graphs}

\begin{abstract}
Using topological circles in the Freudenthal compactification of a graph as infinite cycles, we extend to locally finite graphs a result of Oberly and Sumner on the Hamiltonicity of finite graphs. This answers a question of Stein, and gives a sufficient condition for Hamiltonicity in locally finite graphs.
\end{abstract}

\maketitle

\section{Introduction}
Determining whether a finite graph is Hamiltonian is an active field in graph theory. The problem to decide whether a finite graph is Hamiltonian is difficult. This indicates that it should not be easy to find a necessary and sufficient condition for a graph to be Hamiltonian which can easily be checked. On the other hand, there are a lot of conditions that are either necessary or sufficient for Hamiltonicity. In this paper we consider a sufficient condition in terms of forbidden subgraphs, due to Oberly and Sumner~\cite{ObSu}.

In order to state their theorem, we need two definitions. We call a graph \textit{locally connected} if the neighbourhood of each vertex induces a connected subgraph. A graph is called \textit{claw-free} if it does not contain the graph $K_{1, 3}$ as an induced subgraph. Now the theorem of Oberly and Sumner is as follows:

\begin{theorem}\label{Ob_Su}\cite[Thm.\ 1]{ObSu}
Every finite, connected, locally connected, claw-free graph on at least three vertices is Hamiltonian.
\end{theorem}

Most Hamiltonicity results consider only finite graphs. The reason for this is that it is not obvious what a Hamilton cycle in an infinite graph should be. We follow the topological approach of \cite{diestel_buch, diestel_arx}, which is to take as the infinite cycles of a graph $G$ the circles in its Freudenthal compactification $|G|$. It is then natural to call $G$ \textit{Hamiltonian} if there is a circle in $|G|$ that contains all vertices of $G$.

Based on this notion of Hamiltonicity for locally finite connected graphs, some Hamiltonicity results for finite graphs have already been generalized to locally finite graphs, see \cite{brewster-funk, bruhn-HC, agelos-HC, lehner-HC}. The most natural candidates for Hamiltonicity theorems that might generalize to locally finite graphs are probably those based on a local condition, such as Theorem~\ref{Ob_Su}: such conditions will be well defined also in infinite graphs and tend to be susceptible to compactness arguments.

In \cite[Question 5.1.3]{stein} Stein asks whether Theorem~\ref{Ob_Su} can be generalized to locally finite graphs. We answer this question positively by the following theorem, which is the main result of this paper.

\begin{theorem}\label{Inf_Ob_Su}
Every locally finite, connected, locally connected, claw-free graph on at least three vertices is Hamiltonian.
\end{theorem}

The structure of this paper is as follows. In Section~2 we recall some basic definitions and fix some further notation we shall need in this paper. Section~3 contains some facts, lemmas and theorems which we shall need in the proof of the main result. The proof of Theorem~\ref{Inf_Ob_Su} together with some corollaries is the content of Section~4.

While I was writing this paper, I noticed that Hamann, Lehner and Pott \cite{Ha_Leh_Po} are investigating similar questions.

\section{Basic definitions and notation}

In this section, important definitions and notation are listed. Furthermore, some basic definitions are recalled in order to avoid confusion with regard to the notation. In general, we will follow the graph theoretical notation of \cite{diestel_buch} in this paper if nothing different is stated. For basic facts about graph theory, the reader is also referred to \cite{diestel_buch}. Besides finite graph theory, a topological approach to locally finite graphs is covered in \cite[Ch.\ 8.5]{diestel_buch}. For a wider survey in this field, see \cite{diestel_arx}.

All graphs which are considered in this paper are undirected and simple. In general, we do not assume a graph to be finite. For this section, we fix an arbitrary graph $G = (V, E)$.

The graph $G$ is called \textit{locally finite} if every vertex of $G$ has only finitely many neighbours.

For a vertex set $X$ of $G$, we denote by $G[X]$ the induced subgraph of $G$ with vertex set $X$. For vertex sets with up to four vertices, we omit the set brackets and write $G[v, w, x, y]$ instead of $G[\lbrace v, w, x, y \rbrace]$ where $\lbrace v, w, x, y \rbrace \subseteq V$. We write $G-X$ for the graph $G[V \setminus X]$ and for singleton sets, we omit the set brackets and write just $G-v$ instead of $G-\lbrace v \rbrace$ where $v \in V$. For the cut which consists of all edges of $G$ that have one endvertex in $X$ and the other endvertex in $V \setminus X$, we write $\delta(X)$.

Let $C$ be a cycle of $G$ and $u$ be a vertex of $C$. Then we write $u^+$ and $u^-$ for the neighbour of $u$ in $C$ in positive and negative, respectively, direction of $C$ using a fixed orientation of $C$. Later on we will not mention that we fix an orientation for the considered cycle using this notation. We implicitly fix an arbitrary orientation of the cycle.

Let $P$ be a path in $G$ and $T$ be a tree in $G$. We write $\mathring{P}$ for the subpath of $P$ which we obtain from $P$ by removing the endvertices of $P$. If $s$ and $t$ are vertices of $T$, we write $sTt$ for the unique path in $T$ with endvertices $s$ and $t$. Note that this covers also the case where $T$ is a path. If $P_v = v_0 \ldots v_n$ and $P_w = w_0 \ldots w_k$ are paths in $G$ with $n, k \in \mathbb{N}$ where $v_n$ and $w_0$ may be equal but apart from that these paths are disjoint and the vertices $v_n, w_0$ are the only vertices of $P_v$ and $P_w$ which lie in $T$, we write $v_0 \ldots v_nTw_0 \ldots w_k$ for the path with vertex set ${V(P_v) \cup V(v_nTw_0) \cup V(P_w)}$ and edge set ${E(P_v) \cup E(v_nTw_0) \cup E(P_w)}$.

For a vertex set $X \subseteq V$ and an integer $k \geq 1$, we denote with $N_k(X)$ the set of vertices in $G$ from which the distance is at least $1$ and at most $k$ to $X$ in $G$. We write $N(X)$ instead of $N_1(X)$, which denotes the usual neighbourhood of $X$ in $G$. For a singleton set $\lbrace v \rbrace \subseteq V$, we omit the set brackets and write just $N_k(v)$ and $N(v)$ instead of $N_k(\lbrace v \rbrace)$ and $N(\lbrace v \rbrace)$, respectively. If $H$ is a subgraph of $G$, we just write $N_k(H)$ and $N(H)$ instead of $N_k(V(H))$ and $N(V(H))$, respectively.

We call $G$ \textit{locally connected} if for every vertex $v \in V$ the subgraph $G[N(v)]$ is connected.

We refer to the graph $K_{1, 3}$ also as \textit{claw}. The graph $G$ is called \textit{claw-free} if it does not contain the claw as an induced subgraph.

A one-way infinite path in $G$ is called a \textit{ray} of $G$. Now an equivalence relation can be defined on the set of all rays of $G$ by saying that two rays in $G$ are equivalent if they cannot be separated by finitely many vertices. It is easy to check that this relation really defines an equivalence relation. The corresponding equivalence classes of rays under this relation are called the \textit{ends} of $G$.

For the rest of this section, we assume $G$ to be locally finite and connected. A topology can be defined on $G$ together with its ends to obtain a topological space which we call $|G|$. For a precise definition of $|G|$, see \cite[Ch.\ 8.5]{diestel_buch}. It should be pointed out that, inside $|G|$, every ray of $G$ converges to the end of $G$ it is contained in.

Apart from the definition of $|G|$ as in \cite[Ch.\ 8.5]{diestel_buch}, there is an equivalent way of defining the topological space $|G|$, namely, by endowing $G$ with the topology of a \linebreak $1$-complex (also called CW complex of dimension $1$) and considering the Freudenthal compactification of $G$. This connection was examined in \cite{freud-equi}. For the original paper of Freudenthal about the Freudenthal compactification, see \cite{freud}.

For a point set $X$ in $|G|$, we denote its closure in $|G|$ by $\overline{X}$.

A subspace $Z$ of $|G|$ is called \textit{standard subspace} of $|G|$ if $Z = \overline{H}$ where $H$ is a subgraph of $G$.

A \textit{circle} in $|G|$ is the image of a homeomorphism which maps from the unit circle $S^1$ in $\mathbb{R}^2$ to $|G|$. The graph $G$ is called \textit{Hamiltonian} if there exists a circle in $|G|$ which contains all vertices of $G$. We call such a circle a \textit{Hamilton circle} of $G$. For $G$ being finite, this coincides with the usual meaning, namely that there is a cycle in $G$ which contains all vertices of $G$. Such cycles are called \textit{Hamilton cycles} of $G$.

The image of a homeomorphism which maps from the closed real unit interval $[0, 1]$ to $|G|$ is called an \textit{arc} in $|G|$. For an arc $\alpha$ in $|G|$, we call a point $x$ of $|G|$ an \textit{endpoint} of $\alpha$ if $0$ or $1$ is mapped to $x$ by the homeomorphism which defines $\alpha$. Furthermore, we say that $\alpha$ \textit{ends} in a point $x$ of $|G|$ if $x$ is an endpoint of $\alpha$. A ray together with the end it converges to is a simple example of an arc. In general, the structure of arcs is more complicated. An arc may contain $2^{\aleph_0}$ many ends.

A subspace $Z$ of $|G|$ is called \textit{arc-connected} if for every two points of $Z$ there is an arc in $Z$ which has these two points as its endpoints.

Let $\omega$ be an end of $G$ and $Z$ be a standard subspace of $|G|$. Then we define the \textit{degree} of $\omega$ in $Z$ as a value in $\mathbb{N} \cup \lbrace \infty \rbrace$, namely the supremum of the number of edge-disjoint arcs in $Z$ that end in $\omega$. The next definition is due to Bruhn and Stein (see \cite{circle}) and allows us to distinguish the parity of degrees of ends also in the case where their degrees are infinite. We call the degree of $\omega$ \textit{even} in $Z$ if there exists a finite set $S \subseteq V$ such that for every finite set $S' \subseteq V$ with $S \subseteq S'$ the maximum number of edge-disjoint arcs in $Z$ whose endpoints are $\omega$ and some $s \in S'$ is even. Otherwise, we call the degree of $\omega$ \textit{odd} in $Z$.

\section{Toolkit}

This section covers some facts which we shall need later or for the proofs of the last two lemmas of this section, Lemma~\ref{struct_toll} and Lemma~\ref{HC-extract}. These two lemmas are very important tools for the proof of the main result. We begin with a basic proposition about infinite graphs.

\begin{proposition}\label{ray}\cite[Prop.\ 8.2.1]{diestel_buch}
Every infinite connected graph has a vertex of infinite degree or contains a ray.
\end{proposition}

Especially for infinite connected graphs which are locally finite, Proposition~\ref{ray} shows the existence of rays. The proof of this proposition is based on a compactness argument and is not difficult. Anyhow, we do not state a proof here.

We proceed with a couple of lemmas about connectedness in the topological space $|G|$ of a locally finite connected graph $G$.

\begin{lemma}\label{arc_conn}\cite[Thm.\ 2.6]{path-cyc-tree}
If $G$ is a locally finite connected graph, then every closed topologically connected subset of $|G|$ is arc-connected.
\end{lemma}

It follows from this lemma that being connected is equivalent to being arc-connected for closed topologically connected subsets of $|G|$. We shall use this fact for topologically connected standard subspaces, which are closed by definition.

The next lemma is a basic tool and gives a necessary condition for the existence of certain arcs in $|G|$. To formulate the lemma, we use the following notation. For an edge set $F \subseteq E(G)$, we denote by $\mathring{F}$ the point set in $|G|$ which consists of all inner points of edges of $F$.

\begin{lemma}\label{jumping-arc}\cite[Lemma 8.5.3]{diestel_buch}
Let $G$ be a locally finite connected graph and \linebreak ${F \subseteq E(G)}$ be a cut with the sides $V_1$ and $V_2$.
\textnormal{
\begin{enumerate}[\normalfont(i)]
\item \textit{If $F$ is finite, then $\overline{V_1} \cap \overline{V_2} = \emptyset$, and there is no arc in $|G| \setminus \mathring{F}$ with one endpoint in $V_1$ and the other in $V_2$.}
\item \textit{If $F$ is infinite, then $\overline{V_1} \cap \overline{V_2} \neq \emptyset$, and there may be such an arc.}
\end{enumerate}
}
\end{lemma}

The following lemma states a graph-theoretical characterization of topologically connected standard subspaces. By Lemma~\ref{arc_conn}, we get also a characterization of arc-connected standard subspaces.

\begin{lemma}\label{top_conn}\cite[Lemma 8.5.5]{diestel_buch}
If $G$ is a locally finite connected graph, then a standard subspace of $|G|$ is topologically connected (equivalently: arc-connected) if and only if it contains an edge from every finite cut of $G$ of which it meets both sides.
\end{lemma}

Now we state a theorem which helps us to verify when every vertex and every end of a graph has even degree in a standard subspace.

\begin{theorem}\label{cycspace}\cite[Thm.\ 2.5]{diestel_arx}
Let $G$ be a locally finite connected graph. Then the following are equivalent for $D \subseteq E(G)$:
\begin{enumerate}[\normalfont(i)]
\item $D$ meets every finite cut in an even number of edges.
\item Every vertex and every end of $G$ has even degree in $\overline{D}$.
\end{enumerate}
\end{theorem}

In \cite[Thm.\ 2.5]{diestel_arx} are actually four equivalent statements involved, but we need only two of them here. Theorem~\ref{cycspace} follows from a result of Diestel and \linebreak K\"{u}hn \cite[Thm.\ 7.1]{inf-cyc-1} together with a result of Berger and Bruhn \cite[Thm.\ 5]{even-deg}.

Bruhn and Stein showed the following characterization of circles in terms of vertex and end degrees.

\begin{lemma}\label{circ}\cite[Prop.\ 3]{circle}
Let $C$ be a subgraph of a locally finite connected graph $G$. Then $\overline{C}$ is a circle if and only if $\overline{C}$ is topologically connected and every vertex or end $x$ of $G$ with $x \in \overline{C}$ has degree two in $C$.
\end{lemma}

It should be noted at this point that Theorem~\ref{cycspace} and Lemma~\ref{circ} are crucial for the proof of Lemma~\ref{HC-extract}.

Now we turn towards claw-free graphs and prove two basic facts about minimal vertex separators in such graphs.

\begin{proposition}\label{2 comp}
Let $G$ be a connected claw-free graph and $S$ be a minimal vertex separator in $G$. Then $G-S$ has exactly two components.
\end{proposition}

\begin{proof}
Suppose $G-S$ has at least three components. Since $S$ is a minimal vertex separator, every vertex of $S$ has at least one neighbour in each component of $G-S$. Now pick an $s \in S$ and neighbours $v_1, v_2, v_3$ of $s$ which lie in different components of $G-S$. Then $G[s, v_1, v_2, v_3]$ is an induced claw. This contradicts our assumption on $G$.
\end{proof}

The next lemma is a cornerstone of the constructive proof of Lemma~\ref{Ob_Su-cut-1}.

\begin{lemma}\label{complete}
Let $G$ be a connected claw-free graph and $S$ be a minimal vertex separator in $G$. For every vertex $s \in S$ and every component $K$ of $G-S$, the graph $G[N(s)\cap V(K)]$ is complete.
\end{lemma}

\begin{proof}
By Proposition~\ref{2 comp}, we know that $G-S$ has precisely two components, say $K_1$ and $K_2$. Now suppose for a contradiction that the statement of the lemma is false. Then there exists a vertex $s \in S$ such that $s$ has two distinct neighbours $v_1, v_2$ which are not adjacent and lie both in $K_1$ or $K_2$, say in $K_1$. Since $S$ is a minimal separator, it has at least one neighbour in each component of $G-S$. Let $v_3$ be a neighbour of $s$ in $K_2$. Now the graph $G[s, v_1, v_2, v_3]$ is an induced claw in $G$, which is a contradiction.
\end{proof}

Before we turn towards the two main lemmas of this section, let us prove a basic fact about locally connected graphs.

\begin{proposition}\label{2-conn}
Every connected, locally connected graph on at least three vertices is $2$-connected.
\end{proposition}

\begin{proof}
Let $G$ be a connected, locally connected graph on at least three vertices and suppose it is not $2$-connected. Then $G$ cannot be complete. So there exists a minimal vertex separator which consists just of one vertex, say $s$, because $G$ is connected. By minimality, $s$ has neighbours in all components of $G-s$. Let $v_1$ and $v_2$ be neighbours of $s$ which lie in different components of $G-s$. Since $G$ is locally connected, there exists a path in the neighbourhood of $s$ from $v_1$ to $v_2$. This path does not meet $s$ but connects two components of $G-s$. This contradicts that $\lbrace s \rbrace$ is a separator in $G$.
\end{proof}

The following lemma deals with the structure of an infinite, locally finite, connected, claw-free graph $G$. Roughly speaking, the lemma says that if we separate any finite connected subgraph from all ends by a finite minimal set $\mathscr{S}$, then this separator decomposes into minimal vertex separators each of which has neighbours in precisely two components of $G-\mathscr{S}$, namely in the unique finite component of $G-\mathscr{S}$ and in an infinite component (see Figure~\ref{struct_toll_situation}).

\begin{lemma}\label{struct_toll}
Let $G$ be an infinite, locally finite, connected, claw-free graph and ${X}$ be a finite vertex set of $G$ such that $G[X]$ is connected. Furthermore, let $\mathscr{S} \subseteq V(G)$ be a finite minimal vertex set such that $\mathscr{S} \cap X = \emptyset$ and every ray starting in $X$ has to meet $\mathscr{S}$. Then the following holds:
\begin{enumerate}[\normalfont(i)]

\item $G-\mathscr{S}$ has $k \geq 1$ infinite components $K_1, \ldots, K_k$ and the set $\mathscr{S}$ is the disjoint union of minimal vertex separators $S_1, \ldots, S_k$ in $G$ such that for every $i$ with $1 \leq i \leq k$ each vertex in $S_i$ has a neighbour in $K_j$ if and only if $j=i$.
\item $G-\mathscr{S}$ has precisely one finite component $K_0$. This component contains all vertices of $X$ and every vertex of $\mathscr{S}$ has a neighbour in $K_0$.
\end{enumerate}
\end{lemma}

\begin{proof}
Since $G[X]$ is connected, there must be a component of $G-\mathscr{S}$ that contains all vertices of $X$. Let $K_0$ be this component. We show first that $K_0$ is finite. Suppose not for a contradiction. Then there is a ray in $K_0$ by Proposition~\ref{ray} since $K_0$ is locally finite and connected. Using the connectedness of $K_0$ again, there exists also a ray in $K_0$ that starts in $X$. No ray in $K_0$ meets the set $\mathscr{S}$ because $K_0$ is a component of $G-\mathscr{S}$. This contradicts the definition of $\mathscr{S}$. By the minimality of $\mathscr{S}$, we get that every vertex $s \in \mathscr{S}$ has at least one neighbour in $K_0$. If this would not be the case, then $\mathscr{S}-s$ would be a proper subset of $\mathscr{S}$ which no ray can avoid that starts in $X$. By the same argument, we also know that every vertex $s \in \mathscr{S}$ has at least one neighbour in an infinite component of $G-\mathscr{S}$. Furthermore, we know that every vertex $s \in \mathscr{S}$ can have only two neighbours in different components because $G$ is claw-free. Since $G$ is locally finite and $\mathscr{S}$ is finite, $G-\mathscr{S}$ has only finitely many components. So let $K_1, \ldots, K_k$ be the infinite components of $G-\mathscr{S}$. Using that $G$ is infinite, locally finite and connected, the graph $G$ contains a ray by Proposition~\ref{ray}, which implies that $k \geq 1$ holds. The previous observations ensure that we can partition the set $\mathscr{S}$ into vertex sets $S_1, \ldots, S_k$ where a vertex $s \in \mathscr{S}$ lies in $S_i$ if and only if $s$ has a neighbour in $K_i$ for every $i$ with $1 \leq i \leq k$. This definition implies that $S_i$ is a separator in $G$ for each $i$ with $1 \leq i \leq k$. Using the minimality of $\mathscr{S}$, we obtain furthermore that each set $S_i$ is a minimal vertex separator in $G$. This completes the proof of statement~(i). It remains to check that $K_0$ is the only finite component of $G - \mathscr{S}$. Let us consider an arbitrary vertex $s \in \mathscr{S}$. We know that $s$ lies in $S_i$ for some $i$ with $1 \leq i \leq k$. So $s$ has a neighbour in two components of $G-\mathscr{S}$, namely $K_0$ and $K_i$. Since $G$ is claw-free, the vertex $s$ cannot have a neighbour in any other component of $G-\mathscr{S}$. As $s$ was chosen arbitrarily, we obtain that $K_0, K_1, \ldots, K_k$ are all components of $G-\mathscr{S}$. Especially, $K_0$ is the only finite component of $G-\mathscr{S}$ and contains all vertices of $X$ by definition.
\end{proof}

\begin{figure}[htbp]
\centering
\includegraphics{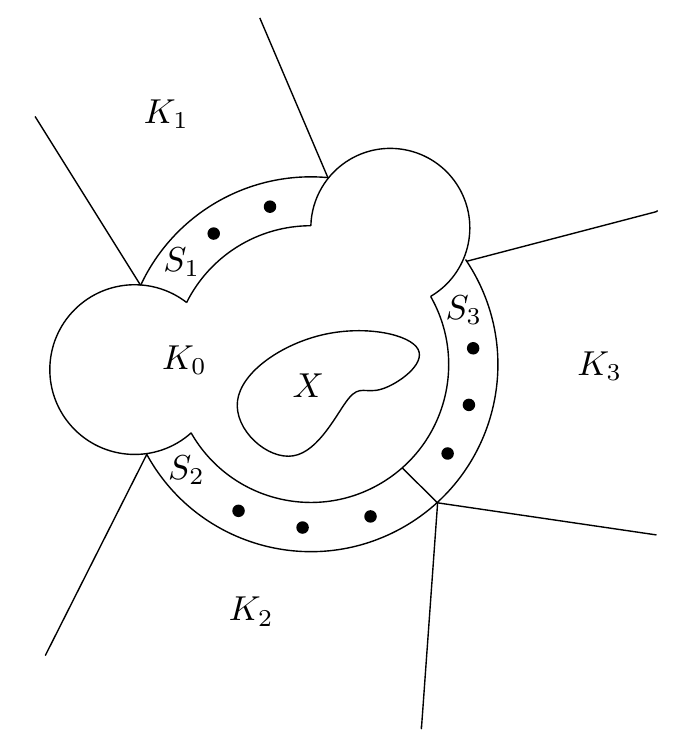}
\caption{Example for Lemma~\ref{struct_toll} with $k=3$.}
\label{struct_toll_situation}
\end{figure}

The next lemma is our main tool to prove that an infinite, locally finite, connected graph $G$ is Hamiltonian. In order to apply this lemma, we need a sequence of cycles of $G$ and a set of vertex sets which fulfil a couple of conditions. While we obtain a Hamilton circle as a limit object from the sequence of cycles, the vertex sets act as witnesses to verify that the limit object is really a circle. Let us look more closely at the idea of the lemma before we state it. We define a limit object from a sequence of cycles of $G$ by saying that a vertex or an edge of $G$ is contained in the limit if it lies in all but finitely many cycles of the sequence. Conversely, we say that a vertex or an edge of $G$ is not in the limit if it lies in only finitely many of the cycles. Since we have to be able to tell for each vertex and for each edge of $G$ whether it shall be contained in the limit or not, we have the conditions $(\text{i})$ and $(\text{iv})$ in the lemma. Condition $(\text{i})$ takes care of even more. It forces every vertex of $G$ to be in the limit, which must be the case if we want to end up with a Hamilton circle. To ensure that the limit object becomes a circle, it is enough by Lemma~\ref{circ} to take care that the limit is topologically connected and that the degree of each vertex and end is two in the limit. Everything except the right end degree is no problem using the cycles in the sequence and condition $(\text{iv})$. However, without further conditions the limit might have ends with degree larger than two. To prevent this problem, we use the conditions $(\text{ii})$, $(\text{iii})$ and $(\text{v})$. They guaranty the existence of a sequence of finite cuts for each end such that each sequence converges to its corresponding end and the limit object meets each of the cuts precisely twice. This bounds the end degree by two. Note that it is not easy to get such cycles and vertex sets. The main work to prove Theorem~\ref{Inf_Ob_Su} is to construct cycles and vertex sets which fulfil the required conditions. The way of constructing these objects relies on the structure of the graph as described in Lemma~\ref{struct_toll}.

\begin{lemma}\label{HC-extract}
Let $G$ be an infinite, locally finite, connected graph and $(C_i)_{i \in \mathbb{N}}$ be a sequence of cycles of $G$. Now $G$ is Hamiltonian if there exists an integer $k_i \geq 1$ for every $i \geq 1$ and vertex sets $M^i_j \subseteq V(G)$ for every $i \geq 1$ and $j$ with $1 \leq j \leq k_i$ such that the following is true:
\textnormal{
\begin{enumerate}[\normalfont(i)]
\item \textit{For every vertex $v$ of $G$, there exists an integer $j \geq 0$ such that $v \in V(C_i)$ holds for every $i \geq j$.}
\item \textit{For every $i \geq 1$ and $j$ with $1 \leq j \leq k_i$, the cut $\delta(M^i_j)$ is finite.}
\item \textit{For every end $\omega$ of $G$, there is a function $f : \mathbb{N} \setminus \lbrace 0 \rbrace \longrightarrow \mathbb{N}$ such that the inclusion ${M^{j}_{f(j)} \subseteq M^i_{f(i)}}$ holds for all integers $i, j$ with $1 \leq i \leq j$ and the equation ${M_{\omega}:= \bigcap^{\infty}_{i=1} \overline{M^i_{f(i)}} = \lbrace \omega \rbrace}$ is true.}
\item \textit{$E(C_i) \cap E(C_j) \subseteq E(C_{j+1})$ holds for all integers $i$ and $j$ with $0 \leq i < j$.}
\item \textit{The equations $E(C_i) \cap \delta(M^p_j) = E(C_p) \cap \delta(M^p_j)$ and $|E(C_i) \cap \delta(M^p_j)| = 2$ hold for each triple $(i, p, j)$ which satisfies $1 \leq p \leq i$ and $1 \leq j \leq k_p$.}
\end{enumerate}
}
\end{lemma}

\begin{proof}
We define a subgraph $C$ of $G$ and show that its closure is a Hamilton circle of $G$. Let
\begin{align*}
V(C) &= \bigcup^{\infty}_{i=0} V(C_i), \\
E(C) &= \lbrace e \in E(G) \; : \; e \in E(C_i) \textnormal{ for infinitely many } i \geq 0 \rbrace.
\end{align*}
Note that condition~(i) implies $V(C) = V(G)$. So the closure $\overline{C}$ contains all ends of $G$. We get also immediately that $E(C) \neq \emptyset$ by condition~(v). Furthermore, condition~(iv) implies that for every edge $e \in E(C)$ there exists an integer $j \geq 0$ such that $e \in E(C_i)$ for every $i \geq j$. In order to prove that $\overline{C}$ is a Hamilton circle of $G$, we want to apply Lemma~\ref{circ}. So we need $\overline{C}$ to be topologically connected and that every vertex and every end of $\overline{C}$ has degree two in $\overline{C}$. We prove both of these statements with two claims. Before we can do this, we need the following claim.

\setcounter{claim}{0}
\begin{claim}
Let $X \subseteq V(G)$ be a finite set of vertices and $D \subseteq E(G)$ be a finite set of edges. Then there exists an integer $j \geq 0$ such that $X \subseteq V(C_i)$ holds for every $i \geq j$ and that each edge of $D$ is either contained in $E(C_i)$ for every $i \geq j$ or not contained in any $E(C_i)$ for $i \geq j$.
\end{claim}

Since $X$ is finite, we can use condition~(i) to find an integer $q$ such that $X \subseteq V(C_i)$ holds for every $i \geq q$. Each edge $e$ of $D$ lies either in at most one cycle $C_{\ell - 1}$ of the sequence with $\ell \geq 1$ or in at least two, say in $C_{m}$ and $C_{n}$ with $m, n \geq 0$. In the first case, we know that $e$ does not lie in any $E(C_i)$ for $i \geq \ell$. Using condition~(iv), we obtain in the latter case that $e$ lies in $E(C_i)$ for every $i \geq \ell'$ where $\ell' = \textnormal{max}\lbrace m, n \rbrace$. Using these observations, we can define a function $g: D \longrightarrow \mathbb{N}$ by
\[g(e) = 
\begin{cases}
\ell  &\mbox{if $\ell$ is the least integer such that $e \notin E(C_i)$ for every $i \geq \ell$} \\
\ell' & \mbox{if $\ell'$ is the least integer such that $e \in E(C_i)$ for every $i \geq \ell'$}.
\end{cases} \]
Since $D$ is finite, we can set $j = \textnormal{max} \lbrace \lbrace q \rbrace \cup \lbrace g(e) \; : \; e \in D \rbrace \rbrace$. The conditions~(i) and (iv) imply that $j$ has the desired properties. This completes the proof of Claim~1.
\newline

Now we state and prove the two claims which allow us to apply Lemma~\ref{circ}.

\begin{claim}
$\overline{C}$ is topologically connected and every vertex as well as every end of $G$ has even degree in $\overline{C}$.
\end{claim}

Since $\overline{C}$ is a standard subspace of $|G|$ and contains all vertices of $G$, it suffices to show, by Lemma~\ref{top_conn} and Theorem~\ref{cycspace}, that $E(C)$ meets every nonempty finite cut in an even number of edges and at least twice. So fix any nonempty finite cut $D$. Now take an integer $j$ such that $C_j$ contains vertices from each side of the partition which induces $D$ and that each edge of $D$ is either contained in $E(C_i)$ for every $i \geq j$ or not contained in any $E(C_i)$ for $i \geq j$. Since $D$ is finite, we obtain by Claim~1 that it is possible to find such an integer $j$. Now we use that the cycle $C_j$ has vertices in both sides of the partition which induces $D$. So it must hit the cut $D$ in an even number of edges and at least twice. The same holds for $C$ by its definition. This completes the proof of Claim~2.

\begin{claim}
Every vertex and every end in $\overline{C}$ has degree two in $\overline{C}$.
\end{claim}

As noted before, $\overline{C}$ contains all vertices and ends of $G$. First, we check that every vertex has degree two in $C$. For this purpose, we fix an arbitrary vertex $v$ of $G$. Let $j$ be an integer such that $v$ is a vertex of $C_j$ and that each edge which is incident with $v$ is either contained in $E(C_i)$ for every $i \geq j$ or not contained in any $E(C_i)$ for $i \geq j$. We can find such an integer $j$ because of Claim~1 and because $G$ is locally finite, which implies that the cut $\delta(\lbrace v \rbrace)$ is finite. Since $C_j$ is a cycle and $v$ is one of its vertices, $v$ has degree two in $C_j$ and therefore also in $C$ by definition.

For the statement about the ends, we have to show that for every end the maximum number of edge-disjoint arcs in $\overline{C}$ ending in this end is two. We prove first that the degree of every end must be at least two in $\overline{C}$. Afterwards, we show that the degree of every end is less than or equal to two in $\overline{C}$.

We already know by Claim~$2$ that the standard subspace $\overline{C}$ is topologically connected. So $\overline{C}$ is also arc-connected by Lemma~\ref{arc_conn}. Therefore, no end has degree zero in $\overline{C}$. Additionally, we know by Claim~$2$ that every end has even degree in $\overline{C}$. Combining these facts, the degree of every end must be at least two in $\overline{C}$.

Now let us fix an arbitrary end $\omega$ of $G$. In order to bound the degree of $\omega$ in $\overline{C}$ from above by two, we use the cuts $\delta(M^i_{j})$. We know by condition~(iii) that there exists a function $f$ such that $M_{\omega} = \lbrace \omega \rbrace$ holds. We prove first that for every arc $\alpha$ in $|G|$ which ends in $\omega$ there exists an integer $i'$ such that $\alpha$ uses an edge of $\delta(M^i_{f(i)})$ for every $i \geq i'$.
It is an easy consequence of Lemma~\ref{jumping-arc} that every arc that does not only consist of inner points of edges must contain a vertex.
So we fix a vertex $v$ of $\alpha$. Now choose $i'$ such that $v$ does not lie in $M^{i}_{f(i)}$ for any $i \geq i'$. This is possible because of condition~(iii). We know by condition~(ii) that the cut $\delta(M^{i}_{f(i)})$ is finite for every $i \geq i'$. Now $\alpha$ is an arc such that $v$ is on one side of the finite cut $\delta(M^{i}_{f(i)})$ and $\omega$ is in the closure of the other side for every $i \geq i'$. So by Lemma~\ref{jumping-arc}, the arc $\alpha$ must use one of the edges of $\delta(M^{i}_{f(i)})$ for every $i \geq i'$. Since $C$ contains only two edges of the cut $\delta(M^{i}_{f(i)})$ for every $i \geq 1$ by condition~(v), there can be at most two edge-disjoint arcs in $\overline{C}$ that end in $\omega$. This completes the proof of Claim~3.
\newline

As mentioned before, we can use Claim~2 and Claim~3 to deduce from Lemma~\ref{circ} that $\overline{C}$ is a circle in $|G|$. Furthermore, $\overline{C}$ is a Hamilton circle since $\overline{C}$ contains all vertices of $G$ as we have seen before.
\end{proof}

\section{Locally connected claw-free graphs}

We begin this section with a lemma that contains the essence of the proof of Theorem~\ref{Ob_Su} which Oberly and Sumner presented in \cite[Thm.\ 1]{ObSu}.

\begin{lemma}\label{Ob_Su-enlarge}
Let $G$ be a locally connected claw-free graph and $C$ be a cycle in $G$. Furthermore, let $v \in N(C)$ and $u \in N(v) \cap V(C)$. Then one of the following holds:
\textnormal{
\begin{enumerate}[\normalfont(i)]
\item \textit{For some $x \in \lbrace u^+, u^- \rbrace$, there exists a $v$--$x$ path $P_x$ in $G$ whose vertices lie completely in $N(u)$ such that $V(P_x) \cap \lbrace u^+, u^- \rbrace = \lbrace x \rbrace$ and for every vertex $z \in V(\mathring{P_x}) \cap V(C)$ the relations $V(P_x) \cap \lbrace z^+, z, z^- \rbrace = \lbrace z \rbrace$ and
$z^+z^- \in E(G)$ hold.}
\item \textit{The vertices $u^+$ and $u^-$ are adjacent in $G$ and there exists a certain vertex ${w \in (N(u) \cap V(C)) \setminus \lbrace u^+, u^- \rbrace}$ together with a $v$--$w$ path $P_w$ in $G$ whose vertices lie completely in $N(u)$ such that $u$ is adjacent to $w^+$ or $w^-$, the vertices $u^+, u^-, w^+, w^-$ do not lie on $P_w$ and for every vertex $q \in V(\mathring{P_w}) \cap V(C)$ the relations $V(P_w) \cap \lbrace q^+, q, q^- \rbrace = \lbrace q \rbrace$ and $q^+q^- \in E(G)$ hold.}
\end{enumerate}
}
\end{lemma}

\begin{proof}
Since $G$ is locally connected, there exists a $v$--$u^+$ path $Q$ whose vertices lie entirely in $N(u)$. Let $x$ be the first vertex on $Q$, in the direction from $v$ to $u^+$, which lies in $\lbrace u^+, u^- \rbrace$. Now we set $P_x = vQx$.

Before we proceed, we define the following notation. A vertex ${z \in V(\mathring{P_x}) \cap V(C)}$ is called \textit{singular} if neither $z^+$ nor $z^-$ is a neighbour of $u$. Note that the vertices $z^+$ and $z^-$ are adjacent in $G$ for every singular vertex ${z \in V(\mathring{P_x}) \cap V(C)}$ because otherwise $G[z, z^+, z^-, u]$ is a claw, which contradicts the assumption on $G$.

Now we distinguish two cases:

\setcounter{case}{0}
\begin{case}
Every vertex in $V(\mathring{P_x}) \cap V(C)$ is singular.
\end{case}

In this case, we know that the edge $z^-z^+$ has to be present for every vertex ${z \in V(\mathring{P_x}) \cap V(C)}$ as noted above. Now the objects $x$ and $P_x$ verify that statement~(i) of the lemma is true.

\begin{case}
There exists a vertex in $V(\mathring{P_x}) \cap V(C)$ which is not singular.
\end{case}

We may assume that $u^-$ and $u^+$ are adjacent because otherwise the edge $vu^-$ or the edge $vu^+$ must exist to avoid that $G[u, v, u^-, u^+]$ is a claw. Every such edge corresponds to another path $P_y$ with $V(\mathring{P_y}) \cap V(C) = \emptyset$ which we could use instead of $P_x$. Hence, we would be done by Case~1.

Let $w$ be the first vertex in $V(\mathring{P_x}) \cap V(C)$ which is not singular by traversing $P_x$ and starting at $v$. Since $w$ lies in $V(\mathring{P_x}) \cap V(C)$, it cannot be equal to $u^+$ or $u^-$. Additionally, $w^+$ or $w^-$ is adjacent to $u$ because $w$ is not singular. Now set $P_w = vP_xw$. Since $V(P_x) \cap \lbrace u^+, u^- \rbrace = \lbrace x \rbrace$ and $x$ is an endvertex of $P_x$, we get that neither $u^+$ nor $u^-$ lies on $P_w$. Furthermore, we get that $V(P_w)$ does neither contain $w^+$ nor $w^-$ and each vertex $q \in V(\mathring{P_w}) \cap V(C)$ is singular because $w$ has been chosen as the first vertex in $V(\mathring{P_x}) \cap V(C)$ which is not singular by traversing $P_x$ and starting at $v$. To prove that statement~(ii) of the lemma holds in this case, it remains to show that the edge $q^+q^-$ is present for every vertex $q \in V(\mathring{P_w}) \cap V(C)$. By the choice of $w$, we know that every vertex $q \in V(\mathring{P_w}) \cap V(C)$ is a singular vertex in $V(\mathring{P_x}) \cap V(C)$. So all required edges are present by the remark from above.
\end{proof}

To each of the statements~(i) and (ii) of Lemma~\ref{Ob_Su-enlarge} corresponds a cycle that contains the vertex $v$ and all vertices of $C$. In statement~(i), we get such a cycle using $C$ where we replace the path $z^-zz^+$ in $C$ by the edge $z^-z^+$ for every vertex $z \in V(\mathring{P_x}) \cap V(C)$ and the edge $ux$ of $C$ by the path $uvP_xx$. In statement~(ii), we take $C$ and replace the path $u^-uu^+$ of $C$ by the edge $u^-u^+$, the edge $yw$ of $C$ by the path $yuvP_ww$ for some $y \in \lbrace w^+, w^- \rbrace \cap N(u)$ and each path $q^-qq^+$ in $C$ by the edge $q^-q^+$ for every vertex $q \in V(\mathring{P_w}) \cap V(C)$. We call each of these resulting cycles a \textit{path extension of} $C$.
A finite sequence of cycles $(C_i)$, where $0 \leq i \leq n$ for some $n \in \mathbb{N}$, is called a \textit{path extension sequence of} $C$ if $C_0 = C$ and $C_i$ is a path extension of $C_{i-1}$ for every $i \in \lbrace 1, \ldots, n \rbrace$. A path $P_x$ or $P_w$ as in statement~(i) or (ii), respectively, of Lemma~\ref{Ob_Su-enlarge} is called \textit{extension path}. The vertices $v$ and $u$ from Lemma~\ref{Ob_Su-enlarge} are called \textit{target} and \textit{base}, respectively, of the path extension (see Figure~\ref{path-extensions}).

\begin{figure}[htbp]
\centering
\includegraphics[width=4.5cm]{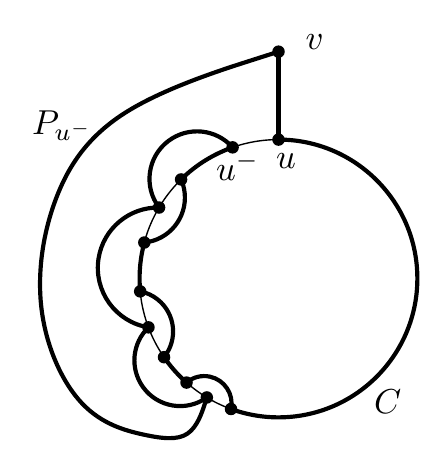}
\hspace{1cm}
\includegraphics[width=4.5cm]{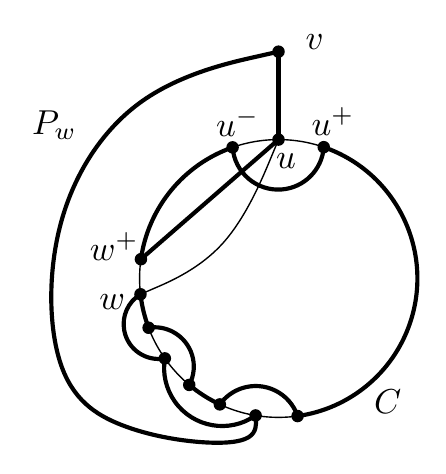}
\caption{Path extensions of a cycle $C$ with target $v$, base $u$ and extension paths $P_{u^-}$ and $P_{w}$, respectively.}
\label{path-extensions}
\end{figure}

Note that Theorem~\ref{Ob_Su} can be easily deduced from Lemma~\ref{Ob_Su-enlarge} together with Proposition~\ref{2-conn}.

Now we turn towards the proof of Theorem~\ref{Inf_Ob_Su}, which is the main result of this paper. The plan to prove this theorem is to construct a sequence of cycles together with certain vertex sets carefully such that we obtain a Hamilton circle as a limit object of the sequence of cycles using Lemma~\ref{HC-extract}. Lemma~\ref{Ob_Su-cut-1} will be the main tool for constructing such objects and relies on path extensions of cycles, which we get as in Lemma~\ref{Ob_Su-enlarge}. Before we turn to Lemma~\ref{Ob_Su-cut-1}, we state a definition which captures when a cycle and vertex sets are appropriate for our purpose and prove a technical lemma about these objects. In the following, we will often deal with sets like $N_k(N(C))$ where $k \geq 1$ and $C$ is some cycle. Therefore, let us recall the definition of sets of this form. The set $N_k(N(C))$ consists of those vertices whose distance to $C$ is at least $2$ and at most $k+1$ together with those vertices which lie on $C$ but whose distance to $N(C)$ is at least $1$ and at most $k$. Note that, in general, $V(C)$ does not have to be a subset of $N_k(N(C))$.

Let $G = (V, E)$ be an infinite locally finite, connected, claw-free graph, $C$ be a cycle of $G$ and $\mathscr{S} \subseteq N(C)$ be a minimal vertex set such that every ray starting in $C$ meets $\mathscr{S}$. Furthermore, let $k$, $S_j$ and $K_j$ be defined as in Lemma~\ref{struct_toll}, $f$ be a permutation of $\lbrace 1, \ldots, k \rbrace$ and $m$ be an integer which satisfies ${0 \leq m \leq k}$. Now we call a tuple $(D, M_{f(1)}, \ldots, M_{f(m)})$ \textit{good} if the following properties hold for every $j \in \lbrace f(1), \ldots, f(m) \rbrace$:
\vspace{4 pt}
\textnormal{
\begin{enumerate}[\normalfont(a)]
\item \textit{$D$ is a cycle of $G$ which contains all vertices of $C$. Moreover, $D$ contains all vertices of $S_{i} \cup (N_3(S_{i}) \cap V(K_{i}))$ for each $i \in \lbrace f(1), \ldots, f(m) \rbrace$.}
\item \textit{${V(K_j) \subseteq M_j \subseteq (V \setminus V(C)) \cup N_2(N(C))}$.}
\item \textit{${|E(D) \cap \delta(M_{j})| = 2}$.}
\item \textit{All vertices of $M_j$ lie on $D$ or in $V \setminus (N_4(K_0) \cup V(K_0))$.}
\item \textit{$G[M_j]$ is connected.}
\item \textit{$M_j$ contains either no or all vertices of $K_p$ for each $p \in \lbrace 1, \ldots, k \rbrace$.}
\end{enumerate}
}
\vspace{4 pt}

We continue with a technical lemma which shows how to get a new good tuple from an old one using path extensions (see Figure~\ref{Ob_Su_cut_pic_1} for a sketch of the situation).

\begin{lemma}\label{good_extensions}
Let $G = (V, E)$ be an infinite locally finite, connected, locally connected, claw-free graph, $C$ be a cycle of $G$ which has a vertex of distance at least $3$ \linebreak to $N(C)$ and $\mathscr{S} \subseteq N(C)$ be a minimal vertex set such that every ray starting in $C$ meets $\mathscr{S}$. Furthermore, let $k$, $S_j$ and $K_j$ be defined as in Lemma~\ref{struct_toll}, $f$ be \linebreak a permutation of $\lbrace 1, \ldots, k \rbrace$ and $m$ be an integer which satisfies ${0 \leq m \leq k}$. If $(D, M_{f(1)}, \ldots, M_{f(m)})$ is a good tuple, then $(D', M'_{f(1)}, \ldots, M'_{f(m)})$ is a good tuple too where $D'$ is a path extension of $D$ with extension path $P_z$ whose \linebreak endvertex different from the target let be $z$ and whose base let be $u$ such that \linebreak ${V(P_z) \cup \lbrace u \rbrace \subseteq V(K_0) \setminus V(C) \cup \mathscr{S} \cup (N_2(N(C)) \cap V(K_0))}$, and $M'_j$ is defined as follows for every $j \in \lbrace f(1), \ldots, f(m) \rbrace$:
\[
M'_{j} = 
\begin{cases}
M_{j} \cup V(P_z) \cup \lbrace u \rbrace &\mbox{if } z \in M_{j} \\
M_{j} \setminus (V(P_z) \cup \lbrace u \rbrace) & \mbox{otherwise}.
\end{cases}
\]
\end{lemma}

\begin{proof}
We fix an arbitrary integer ${j \in \lbrace f(1), \ldots, f(m) \rbrace}$ and check that the tuple ${(D', M'_{f(1)}, \ldots, M'_{f(m)})}$ is good. Property~(a) is valid since ${(D, M_{f(1)}, \ldots, M_{f(m)})}$ is a good tuple and $D'$ is a path extension of $D$.

We use the inclusion ${V(P_z) \cup \lbrace u \rbrace \subseteq V(K_0) \setminus V(C) \cup \mathscr{S} \cup (N_2(N(C)) \cap V(K_0))}$, which holds by assumption, to verify property~(b). Since $\mathscr{S}$ is a subset of $N(C)$ and the inclusion ${V(K_j) \subseteq M_j \subseteq (V \setminus V(C)) \cup N_2(N(C))}$ holds because the tuple ${(D, M_{f(1)}, \ldots, M_{f(m)})}$ is good, the definition of $M'_j$ implies property~(b).

Note for property~(c) that $D'$ has vertices in $M'_j$ and $V \setminus M'_j$. To see this, we use property~(a) and know therefore that $D'$ contains vertices with distance at least $3$ to $N(C)$ in $N_3(\mathscr{S}) \cap V(K_j)$ as well as in $V(C) \setminus N_2(N(C))$. These vertices lie in $M'_j$ and $V \setminus M'_j$, respectively, because of statement~(b). So the cycle $D'$ hits the cut $\delta(M'_j)$ at least twice. Next we show that $D'$ hits $\delta(M'_j)$ at most twice. Note first that neither the edge $uz$ nor any edge of $P_z$ lies in $E(D') \cap \delta(M'_j)$ by definition of the set $M'_j$. Hence, each edge of $E(D') \cap \delta(M'_j)$ is either also an edge of $D$ or of the type $v^{-}v^{+}$ with $v \in V(\mathring{P_z})$ or with $v = u$. Let us look  more closely at such an edge where $v \in V(\mathring{P_z})$. We know that each of the two vertices $v^-$ and $v^+$ lies in $M'_j$ if and only if it lies in $M_j$ by definition of $M'_j$ and since neither of those two vertices lie on $P_z$. So $v^{-}v^{+} \in \delta(M_j)$ is true as well. Therefore, we know that one of the edges $v^{-}v$ and $vv^{+}$ is an element of $E(D) \cap \delta(M_j)$, but it was deleted from $D$ to obtain $D'$ by definition of path extension. A similar observation can be made if $v = u$ using additionally that the edge $u^{-}u^{+}$ is only contained in $D'$ if the extension path $P_z$ contains neither $u^{-}$ nor $u^+$. Putting these observations together, we get that the inequality chain ${|E(D') \cap \delta(M'_j)| \leq |E(D) \cap \delta(M_{j})| = 2}$ holds where the equality follows since ${(D, M_{f(1)}, \ldots, M_{f(m)})}$ is a good tuple.

For property~(d) we use that $D'$ is a path extension of $D$. So it contains all vertices of $D$ and of $V(P_z) \cup \lbrace u \rbrace$. Now the statement follows because all vertices in $M_{j}$ which have distance at most $4$ to $K_0$ are vertices of $D$ by the assumption that ${(D, M_{f(1)}, \ldots, M_{f(m)})}$ is a good tuple.

Property~(e) is obviously valid if $M'_j = M_{j} \cup V(P_z) \cup \lbrace u \rbrace$ using that $G[M_{j}]$ is connected. So let us consider the case where $M'_j = M_{j} \setminus (V(P_z) \cup \lbrace u \rbrace)$. Combining property~(c) and (d), we get that all vertices in $M'_j$ with distance at most $4$ to $K_0$ lie on a path $P$ in $G[M'_j]$ which is induced by the cycle $D'$. We know by assumption that the inclusion ${V(P_z) \cup \lbrace u \rbrace \subseteq V(K_0) \setminus V(C) \cup \mathscr{S} \cup (N_2(N(C)) \cap V(K_0))}$ holds. Hence, all vertices in $N(V(P_z) \cup \lbrace u \rbrace) \cap M'_j$ lie on the path $P$. Using that $G[M_{j}]$ is connected, we know that each component of $G[M'_j]$ contains a vertex of ${N(V(P_z) \cup \lbrace u \rbrace) \cap M'_j}$. Now the path $P$ ensures that $G[M'_j]$ consists of precisely one component, which means that $G[M'_j]$ is connected.

To show that property~(f) holds, we use that the corresponding property with $M_{j}$ instead of $M'_j$ is valid. By the assumption on the extension path, we get the inclusion ${V(P_z) \cup \lbrace u \rbrace \subseteq V(K_0) \setminus V(C) \cup \mathscr{S} \cup (N_2(N(C)) \cap V(K_0)) \subseteq V(K_0) \cup \mathscr{S}}$. Now the definition of $M'_j$ implies that property~(f) holds.
\end{proof}

Using Lemma~\ref{good_extensions} we prove the following lemma which, together with Lemma~\ref{HC-extract}, will be the key for the proof of the main theorem.

\begin{lemma}\label{Ob_Su-cut-1}
Let $G = (V, E)$ be an infinite locally finite, connected, locally connected, claw-free graph, $C$ be a cycle of $G$ which has a vertex of distance at least $3$ to $N(C)$ and $\mathscr{S} \subseteq N(C)$ be a minimal vertex set such that every ray starting in $C$ meets $\mathscr{S}$. Furthermore, let $k$, $S_j$ and $K_j$ be defined as in Lemma~\ref{struct_toll}. Then there exists a cycle $C'$ with the properties:
\textnormal{
\begin{enumerate}[\normalfont(i)]
\item \textit{$V(K_0) \cup \mathscr{S} \cup N_3(\mathscr{S}) \subseteq V(C')$.}
\item \textit{There are vertex sets $M_1, \ldots, M_k \subseteq V$ such that the tuple $(C', M_1, \ldots, M_k)$ is good.}
\item \textit{$E(C - N_2(N(C))) \subseteq E(C')$ and the inclusion $\lbrace u, v \rbrace \subseteq (V \setminus V(C)) \cup N_3(N(C))$ holds for every edge $uv \in E(C') \setminus E(C)$.}
\end{enumerate}
}
\end{lemma}

\begin{proof}
First we define inductively a sequence of $k+1$ cycles $(C_0, \ldots, C_k)$, a bijection ${f: \lbrace 1, \ldots, k \rbrace \longrightarrow \lbrace 1, \ldots, k \rbrace}$ and vertex sets $M^i_j$ for all integers $i \in \lbrace 1, \ldots, k \rbrace$ and $j \in \lbrace f(1), \ldots, f(i) \rbrace$ such that each cycle $C_i$ contains no vertices from any $S_{q}$ for ${q \in \lbrace 1, \ldots, k \rbrace \setminus \lbrace f(1), \ldots, f(i) \rbrace}$ and each tuple ${(C_i, M^i_{f(1)}, \ldots M^i_{f(i)})}$ is good. We start by setting $C_0 = C$. The tuple which consists only of $C$ and does not contain any further vertex sets is good because no requirements have to be fulfilled except that $C_0$ is a cycle that contains all vertices of $C$. Since $\mathscr{S} \subseteq N(C)$ holds by definition, we have checked everything for $i = 0$.

Now suppose we have already defined the sequence of cycles up to $C_m$ with $0 \leq m < k$, the values of $f(i)$ for every $i \in \lbrace 1, \ldots, m \rbrace$ and the sets $M^i_{j}$ for all $i \in \lbrace 1, \ldots, m \rbrace$ and ${j \in \lbrace f(1), \ldots, f(m) \rbrace}$. The definitions of the cycle $C_{m+1}$, of the value of $f(m+1)$ and of the sets $M^{m+1}_j$ for each $j \in \lbrace f(1), \ldots, f(m+1) \rbrace$ need some work. We state and prove two claims before we define these objects.

\setcounter{claim}{0}
\begin{claim}
There are an integer $\ell \in \lbrace 1, \ldots, k \rbrace \setminus \lbrace f(1), \ldots, f(m) \rbrace$, a path extension $D_1$ of $C_m$ which contains precisely one vertex from $S_{\ell}$ but no vertices from any $S_p$ with $p \in \lbrace 1, \ldots, k \rbrace \setminus \lbrace f(1), \ldots, f(m), \ell \rbrace$ and vertex sets $A_j$ for every ${j \in \lbrace f(1), \ldots, f(m) \rbrace}$ such that ${(D_1, A_{f(1)}, \ldots, A_{f(m)})}$ is a good tuple.
\end{claim}

First we pick a vertex $v \in S_{\ell'}$ as target of a path extension of $C_m$ where ${\ell' \in \lbrace 1, \ldots, k \rbrace \setminus \lbrace f(1), \ldots, f(m) \rbrace}$, which is possible since $\mathscr{S} \subseteq N(C)$. Let the vertex ${u \in V(C)}$ be the base of the extension and $P_z$ be the corresponding extension path with endvertices $v$ and $z$. Furthermore, let $s$ be the last vertex on $P_z$ which lies in $\mathscr{S} \setminus \bigcup^m_{i=1} S_{f(i)}$, say $s \in S_{\ell}$ with $\ell \in \lbrace 1, \ldots, k \rbrace \setminus \lbrace f(1), \ldots, f(m) \rbrace$. Now we consider the path extension $D_1$ of $C_m$ where we choose $u$ again as base but with $s$ as target together with the path $P'_{z} = sP_zz$ which we use as extension path (see Figure~\ref{Ob_Su_cut_pic_1}). This way it is ensured that $D_1$ contains precisely one vertex from $\mathscr{S} \setminus \bigcup^m_{i=1} S_{f(i)}$, \linebreak namely $s$.

\begin{figure}[htbp]
\centering
\includegraphics[width=\textwidth]{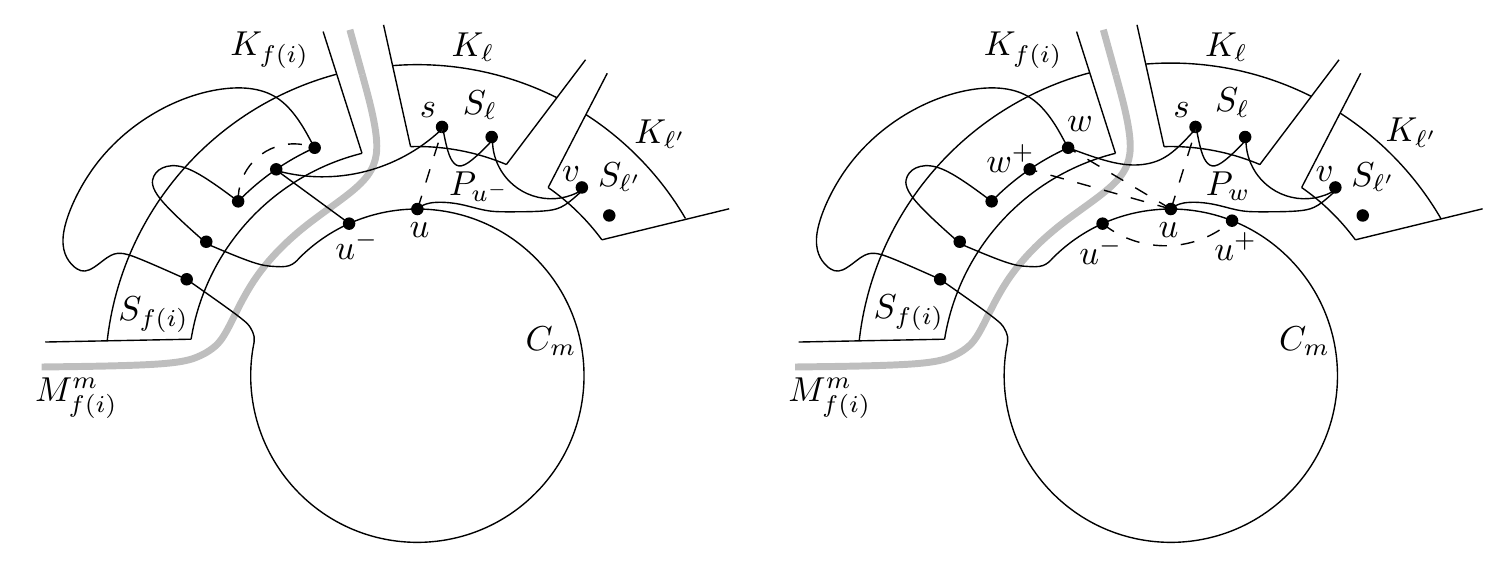}
\caption{Application of Lemma~\ref{good_extensions} with the good tuple $(C_m, M^m_{f(1)}, \ldots, M^m_{f(m)})$ and with the path $sP_{u^-}u^-$ or $sP_{w}w$ as extension path whose base and target are $u$ and $s$, respectively.}
\label{Ob_Su_cut_pic_1}
\end{figure}

To show that we get a good tuple, we apply Lemma~\ref{good_extensions} and define the sets $A_j$ for every integer ${j \in \lbrace f(1), \ldots, f(m) \rbrace}$ also as in Lemma~\ref{good_extensions} using $M^m_j$ and $P'_{z}$. This completes the proof of Claim~1. Note that the application of Lemma~\ref{good_extensions} is possible since $u$ was chosen from ${V(C) \subseteq V(K_0)}$ and has a neighbour in $\mathscr{S} \subseteq N(C)$. With the inclusion ${V(P'_z) \subseteq N(u)}$ we get that ${V(P'_z) \cup \lbrace u \rbrace \subseteq \mathscr{S} \cup (N_2(N(C)) \cap V(K_0))}$ holds.
\newline

For the rest of the proof of the lemma, we fix an integer $\ell$, a cycle $D_1$ and vertex sets $A_j$ for every ${j \in \lbrace f(1), \ldots, f(m) \rbrace}$ with the properties as in Claim~1. Now we proceed with the next claim, which uses these objects.

\begin{claim}
There are a path extension $D_2$ of $D_1$ which contains precisely two vertices from $S_{\ell}$ and these vertices are adjacent in $D_2$ but contains no vertices from any $S_p$ where $p$ lies in $\lbrace 1, \ldots, k \rbrace \setminus \lbrace f(1), \ldots, f(m), \ell \rbrace$ and vertex sets $B_j$ for every ${j \in \lbrace f(1), \ldots, f(m) \rbrace}$ such that ${(D_2, B_{f(1)}, \ldots, B_{f(m)})}$ is a good tuple and each $B_j$ contains either two vertices of $S_{\ell} \cap V(D_2)$ or no vertex of $S_{\ell} \cup V(K_{\ell})$.
\end{claim}

Let $s$ be the only vertex of $D_1$ which lies in $S_{\ell}$. Now we pick a neighbour $v$ of $s$ in $V(K_{\ell})$. There is one because $S_{\ell}$ is a minimal separator in $G$ and $K_{\ell}$ is one of the components of $G-S_{\ell}$ by Lemma~\ref{struct_toll}. Let $P_z$ be the extension path of a path extension of $D_1$ with target $v$ and base $s$ and let $v$ and $z$ be the endvertices of $P_z$. The path $P_z$ must contain vertices in $S_{\ell}$ because it starts in $v \in V(K_{\ell})$ and ends in $z$, which lies in another component of $G-S_{\ell}$ since $D_1$ contains only the vertex $s$ from $S_{\ell}$ and $V(P_z) \subseteq N(s)$. So let $t$ be the last vertex on $P_z$ starting at $v$ which is an element of $S_{\ell}$. Furthermore, let $w$ be the vertex after $t$ on $tP_zz$. By the choice of $t$, we know that $w$ and $z$ lie in the same component of $G-S_{\ell}$. Since $s$ lies in the minimal separator $S_{\ell}$ and $V(P_z) \subseteq N(s)$, we get by Lemma~\ref{complete} that $w$ and $z$ are adjacent. Now we define $D_2$ where we distinguish two cases.

\setcounter{case}{0}
\begin{case}
The vertex $w$ lies in $S_{j'}$ for some $j' \in \lbrace 1, \ldots, k \rbrace \setminus \lbrace f(1), \ldots, f(m) \rbrace$.
\end{case}

Since $D_1$ contains no vertices from $S_{j'}$, we know that $t$ and $z$ lie in $N(w)$ and in the same component of $G-S_{j'}$. So we get by Lemma~\ref{complete} that $t$ and $z$ are adjacent. We take the path $P'_z = tz$ as extension path of a path extension of $D_1$ with target $t$ and base $s$. Furthermore, we set this path extension of $D_1$ to be $D_2$ (see Figure~\ref{Ob_Su_cut_pic_2}). Then $s$ and $t$ are the only vertices from $S_{\ell}$ in $D_2$ and they are adjacent in $D_2$ too. Additionally, we get that $D_2$ contains no vertices from any $S_p$ where $p$ lies in $\lbrace 1, \ldots, k \rbrace \setminus \lbrace f(1), \ldots, f(m), \ell \rbrace$ because of the construction of $D_2$ and because $D_1$ has this property too. This completes the definition of the cycle $D_2$ in Case~1.

\begin{figure}[htbp]
\centering
\includegraphics[width=0.75\textwidth]{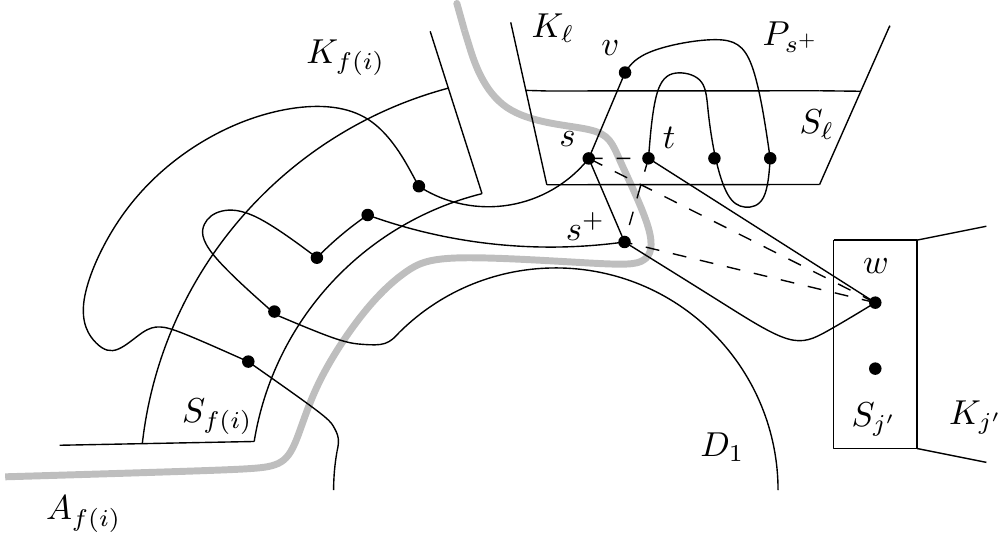}
\caption{Situation in Case~1 of Claim~2.}
\label{Ob_Su_cut_pic_2}
\end{figure}

\begin{case}
The vertex $w$ does not lie in $S_{j}$ for any $j \in \lbrace 1, \ldots, k \rbrace \setminus \lbrace f(1), \ldots, f(m) \rbrace$.
\end{case}

Here we take the path $P'_z = twz$ as extension path of a path extension of $D_1$ with target $t$ and base $s$. We set this path extension of $D_1$ to be $D_2$. Note that $w$ cannot be a vertex of $S_{\ell}$ due to its choice. So we get that $s$ and $t$ are the only vertices from $S_{\ell}$ in $D_2$ and they are adjacent in $D_2$ too. Here we get that $D_2$ contains no vertices from any $S_p$ where $p$ lies in $\lbrace 1, \ldots, k \rbrace \setminus \lbrace f(1), \ldots, f(m), \ell \rbrace$ using the assumption on $w$ together with the construction of $D_2$ and that $D_1$ has this property as well. With this we complete the definition of $D_2$ for Case~2.
\newline

It remains to define the sets $B_j$ for every ${j \in \lbrace f(1), \ldots, f(m) \rbrace}$ and to check that all requirements are fulfilled. We do this for both cases in one step where we use the different definitions of the extension path $P'_z$. We define $B_j$ as in Lemma~\ref{good_extensions} using $A_j$ and $P'_z$. In order to show that ${(D_2, B_{f(1)}, \ldots, B_{f(m)})}$ is a good tuple, we want to apply Lemma~\ref{good_extensions} again. We can do this because the definition of $P'_z$ ensures in each of both cases that the inclusion $V(P'_z) \cup \lbrace s \rbrace \subseteq \mathscr{S} \cup (N(N(C)) \cap V(K_0))$ holds.

The last thing we have to check is that each set $B_j$ contains either two vertices of $S_{\ell} \cap V(D_2)$ or no vertex of $S_{\ell} \cup V(K_{\ell})$. So let us fix an arbitrary ${j \in \lbrace f(1), \ldots, f(m) \rbrace}$. Note that $B_j$ cannot contain vertices from $S_{\ell} \cup (N_3(S_{\ell}) \cap V(K_{\ell}))$ except $s$ and $t$ because each such vertex has distance at most $4$ to $K_0$ and must therefore lie on the cycle $D_2$ by property~(d) of a good tuple. This is not possible since $D_2$ contains only $s$ and $t$ from $S_{\ell}$. The fact that $G[B_j]$ is connected by property~(e) of a good tuple and contains no vertices from $S_{\ell} \cup (N_3(S_{\ell}) \cap V(K_{\ell}))$ except $s$ and $t$ implies that $B_j$ does not contain any of the vertices in $S_{\ell} \cup V(K_{\ell}) \setminus \lbrace s, t \rbrace$. The definition of $B_j$ ensures furthermore that $s$ lies in $B_j$ if and only if $t$ lies in $B_j$. This completes the proof of Claim~2.
\newline

Let us fix a cycle $D_2$ and vertex sets $B_j$ for every ${j \in \lbrace f(1), \ldots, f(m) \rbrace}$ as in Claim~2 for the rest of the proof. Now we are able to define the desired cycle $C_{m+1}$ together with the value $f(m+1)$ and the vertex sets $M^{m+1}_j$ for every ${j \in \lbrace f(1), \ldots, f(m+1) \rbrace}$. We begin by setting
\[ f(m+1) = \ell.\]
Before we build the cycle $C_{m+1}$, we take a finite tree $T_{\ell}$ in $K_{\ell}$ such that the inclusion ${N_3(S_{\ell}) \cap V(K_{\ell}) \subseteq V(T_{\ell})}$ holds. This is possible because $K_{\ell}$ is connected and $G$ is locally finite, which implies together with the finiteness of $S_{\ell}$ that $N_3(S_{\ell})$ is finite. To build the cycle $C_{m+1}$, we take first $D_2$ and replace the edge $st$ of $D_2$ with $s, t \in S_{\ell}$ by the path $sn_sT_{\ell}n_tt$ where $n_s$ and $n_t$ are vertices in $N(s) \cap V(T_{\ell})$ and $N(t) \cap V(T_{\ell})$, respectively. Let us call the resulting cycle $\tilde{C}$. Now we build a path extension sequence $(\tilde{C}_i)$ of $\tilde{C}$ where we choose the targets always from $S_{\ell} \cup V(T_{\ell})$ and the bases always from $V(T_{\ell})$ until a cycle in this sequence contains all vertices of $S_{\ell} \cup V(T_{\ell})$. This is possible by Lemma~\ref{Ob_Su-enlarge} and because $S_{\ell} \cup V(T_{\ell})$ is finite. Let $\tilde{C}_n$ be the last cycle of the sequence. Then we set
\[ C_{m+1} = \tilde{C}_n.\]
Since $V(C_m) \cup S_{\ell} \cup V(T_{\ell}) \subseteq V(C_{m+1})$ holds by construction, we obtain that
\[ V(C) \cup \bigcup^{m+1}_{p=1} (S_{f(p)} \cup (N_3(S_{f(p)}) \cap V(K_{f(p)}))) \subseteq V(C_{m+1}) \]
is true, which is one of the desired properties. Moreover, we get that $C_{m+1}$ does not contain vertices from $S_q$ for any $q \in \lbrace 1, \ldots, k \rbrace \setminus \lbrace f(1), \ldots, f(m+1) \rbrace$. To see this, note that all vertices of $D_2$ which lie in $\mathscr{S}$ are contained in some $S_j$ where $j \in \lbrace f(1), \ldots, f(m+1) \rbrace$ by definition of $D_2$ and Claim~2. Since $\tilde{C}$ does not contain any other vertices from $\mathscr{S}$ than $D_2$ and we choose the bases of the path extensions for the sequence $(\tilde{C}_i)$ always from $V(T_{\ell}) \subseteq V(K_{\ell})$, all vertices of $C_{m+1}$ which lie in $\mathscr{S}$ must also be contained in some $S_j$ where $j \in \lbrace f(1), \ldots, f(m+1) \rbrace$.

Next we define the sets $M^{m+1}_j$ for every $j \in \lbrace f(1), \ldots, f(m+1) \rbrace$ and verify that they have the desired properties. We begin by setting
\[ M^{m+1}_{f(m+1)} = S_{\ell} \cup V(K_{\ell}). \]
It is obvious that the inclusions ${V(K_{\ell}) \subseteq M^{m+1}_{{f(m+1)}} \subseteq (V \setminus V(C)) \cup N_2(N(C))}$ hold and that $G[M^{m+1}_{f(m+1)}]$ is connected. Furthermore, the definition of $M^{m+1}_{f(m+1)}$ implies that $M^{m+1}_{f(m+1)}$ contains either no or all vertices of $K_p$ for every ${p \in \lbrace 1, \ldots, k \rbrace}$. Since $V(C_{m+1})$ contains all vertices of the set $S_{\ell} \cup V(T_{\ell})$ and $N_3(S_{\ell}) \cap V(K_{\ell})$ is a subset of $V(T_{\ell})$, we get that all vertices in $M^{m+1}_{f(m+1)}$ with distance at most $4$ to $K_0$ lie on $C_{m+1}$. It remains to check that the equation ${|E(C_{m+1}) \cap \delta(M^{m+1}_{f(m+1)})| = 2}$ holds. Note that ${|E(D_2) \cap \delta(M^{m+1}_{f(m+1)})| = 2}$ is true because $D_2$ contains only two vertices of $S_{\ell}$ and these vertices are adjacent in $D_2$ by Claim~2. The construction of $C_{m+1}$ ensures that all edges in $E(C_{m+1}) \setminus E(D_2)$ lie in $G[M^{m+1}_{f(m+1)}]$. So only two edges of $C_{m+1}$ meet the cut $\delta(M^{m+1}_{f(m+1)})$.

By Claim~2, we know that $D_2$ contains precisely two vertices $s, t$ which lie in $S_{\ell}$. We know furthermore by Claim~2 that either $s$ and $t$ or none of them is an element of $B_j$ for every ${j \in \lbrace f(1), \ldots, f(m) \rbrace}$. Now we make the following definition for every ${j \in \lbrace f(1), \ldots, f(m) \rbrace}$:
\[M^{m+1}_{j} = 
\begin{cases}
B_{j} \cup S_{\ell} \cup V(K_{\ell}) &\mbox{if } s, t \in B_j \\
B_{j} & \mbox{otherwise}.
\end{cases} \]
Let us fix an arbitrary ${j \in \lbrace f(1), \ldots, f(m) \rbrace}$ and verify the desired properties. It is obvious that $G[M^{m+1}_{j}]$ is connected because $G[B_{j}]$ is connected by Claim~2 and property~(e) of a good tuple. It is also easy to see that the two inclusions ${V(K_{j}) \subseteq M^{m+1}_{j} \subseteq (V \setminus V(C)) \cup N_2(N(C))}$ are true since the corresponding result with $B_j$ instead of $M^{m+1}_{j}$ holds by Claim~2 and property~(b) of a good tuple.

To show that all vertices in $M^{m+1}_{j}$ which have distance at most $4$ to $K_0$ lie on $C_{m+1}$, it suffices to check the case where $M^{m+1}_{j} = B_{j} \cup S_{\ell} \cup V(K_{\ell})$ by Claim~2 and property~(d) of a good tuple. Since all vertices in $B_{j}$ which have distance at most $4$ to $K_0$ lie on $C_{m+1}$ as well as all vertices in $S_{\ell} \cup V(T_{\ell})$ where $N_3(S_{\ell}) \cap V(K_{\ell}) \subseteq V(T_{\ell})$, we get that all vertices in $M^{m+1}_{j}$ with distance at most $4$ to $K_0$ lie on $C_{m+1}$.

Now we check that the cycle $C_{m+1}$ meets precisely two edges of the cut $\delta(M^{m+1}_{j})$. In the case where $M^{m+1}_{j} = B_j$ holds, we know that $B_j$ does not contain any vertices of $S_{\ell} \cup V(K_{\ell})$ because of Claim~2. The equation ${|E(D_2) \cap \delta(B_j)| = 2}$ is also true by Claim~2 and property~(c) of a good tuple. Since all edges of $E(C_{m+1}) \setminus E(D_2)$ lie in $G[S_{\ell} \cup V(K_{\ell})]$, the cycle $C_{m+1}$ hits still precisely two edges of the cut $\delta(B_j)$. So let us consider the case where $M^{m+1}_{j} = B_{j} \cup S_{\ell} \cup V(K_{\ell})$. Here we know by Claim~2 that $B_j$ contains the vertices $s$ and $t$ from $S_{\ell} \cup V(K_{\ell})$ but no other vertex of this set. By Claim~2 and property~(c) of a good tuple, we know additionally that the equation ${|E(D_2) \cap \delta(B_j)| = 2}$ holds. In this case, the equation implies that also ${|E(D_2) \cap \delta(M^{m+1}_{j})| = 2}$ is true. Since all edges of $E(C_{m+1}) \setminus E(D_2)$ lie completely in $G[M^{m+1}_{j}]$, the cycle $C_{m+1}$ meets the cut $\delta(M^{m+1}_{j})$ precisely twice.

It remains to prove that $M^{m+1}_{j}$ has the property that it contains either no or all vertices of $K_p$ for each $p \in \lbrace 1, \ldots, k \rbrace$. We know that this is true for $B_{j}$ by Claim~2 together with property~(e) of a good tuple and so the definition of $M^{m+1}_{j}$ implies again that $M^{m+1}_{j}$ has this property. So the tuple ${(C_{m+1}, M^{m+1}_{f(1)}, \ldots M^{m+1}_{f(m+1)})}$ has all desired properties. This completes the construction of the sequence $(C_0, \ldots, C_k)$ of cycles.
\newline

To complete the proof of this lemma, we take now a path extension sequence $(\hat{C}_i)$ of $C_k$ where we choose the targets and bases always from $V(K_0)$ until a cycle in this sequence contains all vertices of $K_0$. Note that $V(K_0)$ is finite by its choice and Lemma~\ref{struct_toll}. Let $\hat{C}_{M}$ be the last cycle of the sequence. Then we set $C' = \hat{C}_{M}$. The construction of the cycle $C_k$ ensures that $V(C) \cup \mathscr{S} \cup (N_3(\mathscr{S}) \setminus V(K_0)) \subseteq V(C_k)$ holds. Hence, the inclusion $V(K_0) \cup \mathscr{S} \cup N_3(\mathscr{S}) \subseteq V(C')$ is true as desired for statement~(i) of this lemma.

For each cycle $\hat{C}_i$, we define now vertex sets $\hat{M}^i_j$ for every ${j \in \lbrace 1, \ldots, k \rbrace}$ such that each tuple ${(\hat{C}_i, \hat{M}^i_1, \ldots, \hat{M}^i_k)}$ is good. If we have constructed these vertex sets, statement~(ii) of this lemma holds by setting ${M_j = \hat{M}^{M}_j}$ for every ${j \in \lbrace 1, \ldots, k \rbrace}$. For $i = 0$, the construction of $C_k$ ensures that $\hat{M}^0_j = M^{k}_{j}$ is a valid choice for every ${j \in \lbrace 1, \ldots, k \rbrace}$.

Now assume we have already defined for every cycle $\hat{C}_i$ with $0 \leq i \leq N < M$ the corresponding vertex sets $\hat{M}^i_j$. Let $P_{N}$ be the extension path of the path extension $\hat{C}_{N+1}$ of $\hat{C}_{N}$ with base $u$ and endvertex $x$ which is different from the target. We define the set $\hat{M}^{N+1}_j$ as in Lemma~\ref{good_extensions} for every ${j \in \lbrace 1, \ldots, k \rbrace}$ using $\hat{M}^{N}_j$ and $P_N$ together with the base $u$. With this definition, we only have to check that we can apply Lemma~\ref{good_extensions}, which then ensures that the tuple ${(\hat{C}_i, \hat{M}^i_1, \ldots, \hat{M}^i_k)}$ is good. Since $V(P_N)$ is a subset of $N(u)$ and $u$ lies in $K_0$ but has a neighbour, the target of the path extension, which lies in $V(K_0) \setminus V(C)$, we obtain that the inclusion $V(P_{N}) \cup \lbrace u \rbrace \subseteq V(K_0) \setminus V(C) \cup \mathscr{S} \cup (N_2(N(C)) \cap V(K_0))$ holds. Hence, Lemma~\ref{good_extensions} is applicable and so the construction of the path extension sequence $(\hat{C}_{i})$ is done. As mentioned before, this finishes the proof of statement~(ii) of this lemma.
\newline

Finally, we have to show that statement~(iii) of the lemma is true. Note that we have only lost edges of the cycle $C$ by building path extensions. Edges of some cycle we lose by building a path extension of this cycle have always a neighbour on the extension path or are incident with the base of the path extension. Since the bases we have chosen have always a neighbour which does not lie on $C$ and each extension path lies in the neighbourhood of the corresponding base, we obtain that all edges of $C - N_2(N(C))$ must also be edges of $C'$.

Note for the other part of statement~(iii) that each edge $e = uv \in E(C') \setminus E(C)$ lies either on a path whose endvertices are in $\mathscr{S}$ and whose inner vertices lie in some of the trees $T_j$ with $1 \leq j \leq k$ or lies on a path extension we have built during the construction of $C'$. In the first case, the inclusion $\lbrace u, v \rbrace \subseteq V \setminus V(C)$ is valid. So let us consider the second case. It is easy to see that for any cycle $Z$ and any path extension $Z'$ of $Z$ with base $b$ each edge $f = v_1v_2 \in E(Z') \setminus E(Z)$ satisfies $\lbrace v_1, v_2 \rbrace \subseteq N_2(b)$. Either the edge $f$ lies on the corresponding extension path which lies in the neighbourhood of $b$ or the vertices $v_1$ and $v_2$ are the two neighbours of $b$ or of some neighbour of $b$ in $Z$. Now note that during the construction of $C'$ we have always chosen the bases of the path extensions from the set $N(N(C)) \cup V \setminus V(C)$. So the inclusion $\lbrace u, v \rbrace \subseteq N_3(N(C)) \cup V \setminus V(C)$ holds, which proves the other part of statement~(iii) and completes the proof of the lemma.
\end{proof}

Now we are able to prove Theorem~\ref{Inf_Ob_Su}. As remarked earlier, we want to apply Lemma~\ref{HC-extract} to prove this theorem. For this purpose, we will use Lemma~\ref{Ob_Su-cut-1} to obtain a sequence of cycles and vertex sets such that all conditions for the use of Lemma~\ref{HC-extract} are fulfilled.

\begin{proof}[Proof of Theorem~\ref{Inf_Ob_Su}]
Let $G = (V, E)$ be a locally finite, connected, locally connected, claw-free graph on at least three vertices. For the proof, we may assume further that $V$ is infinite because Theorem~\ref{Ob_Su} deals with the finite case.

We want to define a sequence $(C_i)_{i \in \mathbb{N}}$ of cycles of $G$ where each cycle $C_i$ has a vertex with distance at least $3$ to $N(C_i)$. Furthermore, we define an integer sequence $(k_i)_{i \in \mathbb{N} \setminus \lbrace 0 \rbrace}$ and vertex sets $M^i_j \subseteq V(G)$ for every $i \in \mathbb{N} \setminus \lbrace 0 \rbrace$ and $j$ with $1 \leq j \leq k_i$.

We start by taking a finite cycle $\tilde{C}$ in $G$. There exists one since $G$ is $2$-connected by Proposition~\ref{2-conn}. Now we build a path extension sequence of $\tilde{C}$ where we choose the targets always from $N_2(\tilde{C})$ until a cycle in this sequence contains all vertices of $V(\tilde{C}) \cup N_2(\tilde{C})$. This is possible by Lemma~\ref{Ob_Su-enlarge} and since $G$ is locally finite, which implies that $V(\tilde{C}) \cup N_2(\tilde{C})$ contains only finitely many vertices. We define $C_0$ to be the last cycle in the sequence.

Now assume that we have already defined the sequence of cycles up to the cycle $C_m$ for some $m \geq 0$ together with the integer sequence up to $k_m$ and the vertex sets $M^i_j$ for every $i \leq m$ where $j$ satisfies always $1 \leq j \leq k_i$. Then let $\mathscr{S}^{m+1} \subseteq N(C_{m})$ be a finite minimal vertex set such that every ray starting in $V(C_m)$ must meet $\mathscr{S}^{m+1}$. Such a set exists because $G$ is locally finite, which implies that $N(C_{m})$ is finite. Hence, we could get $\mathscr{S}^{m+1}$ by sorting out vertices from $N(C_m)$. Next we set $k_{m+1}$ as the integer we get from Lemma~\ref{struct_toll}. Furthermore, let $S^{m+1}_1, \ldots, S^{m+1}_{k_{m+1}}$ be the minimal separators and $K^{m+1}_0, \ldots, K^{m+1}_{k_{m+1}}$ be the components of $G-\mathscr{S}^{m+1}$ which we get from Lemma~\ref{struct_toll}. With these objects and the cycle $C_m$, we can apply Lemma~\ref{Ob_Su-cut-1} and obtain a new cycle which we set as $C_{m+1}$ and vertex sets for every $j$, which we set as $M^{m+1}_j$ where $1 \leq j \leq k_{m+1}$ holds.

\begin{figure}[htbp]
\centering
\includegraphics[width=0.65\textwidth]{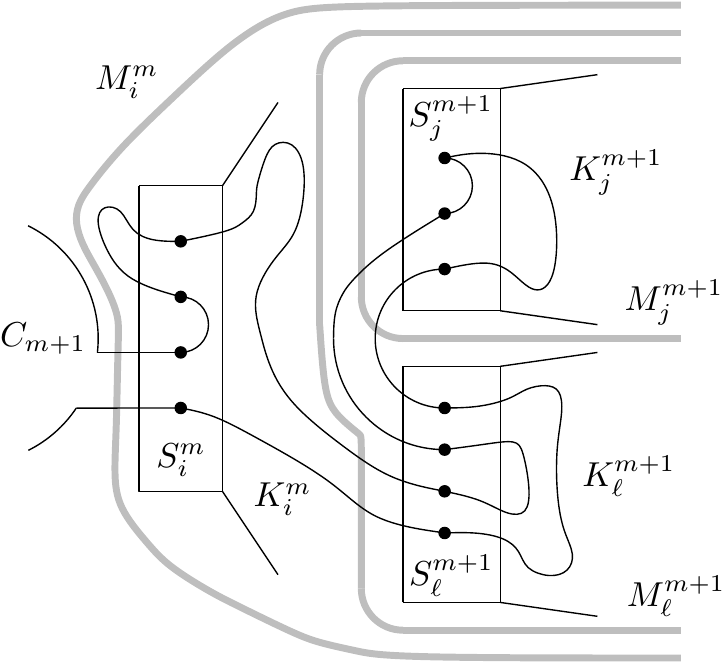}
\caption{The cycle $C_{m+1}$ together with the vertex sets from Lemma~\ref{Ob_Su-cut-1}.}
\label{Thm_1_2_cycle_seq}
\end{figure}

In order to prove that $G$ is Hamiltonian, we want to use Lemma~\ref{HC-extract}. For this purpose, we show the following claim which tells us that the cycles and vertex sets are arranged in a correct way (see Figure~\ref{Thm_1_2_cycle_seq}).

\setcounter{claim}{0}
\begin{claim}
\textnormal{
\begin{enumerate}[\normalfont(a)]
\item \textit{For every vertex $v$ of $G$, there exists an integer $j \geq 0$ such that $v \in V(C_i)$ holds for every $i \geq j$.}
\item \textit{For every $i \geq 1$ and $j$ with $1 \leq j \leq k_i$, the cut $\delta(M^i_j)$ is finite.}
\item \textit{For every end $\omega$ of $G$, there is a function $f : \mathbb{N} \setminus \lbrace 0 \rbrace \longrightarrow \mathbb{N}$ such that the inclusion ${M^{j}_{f(j)} \subseteq M^i_{f(i)}}$ holds for all integers $i, j$ with $1 \leq i \leq j$ and the equation ${M_{\omega}:= \bigcap^{\infty}_{i=1} \overline{M^i_{f(i)}} = \lbrace \omega \rbrace}$ is true.}
\item \textit{$E(C_i) \cap E(C_j) \subseteq E(C_{j+1})$ holds for all integers $i$ and $j$ with $0 \leq i < j$.}
\item \textit{The equations $E(C_i) \cap \delta(M^p_j) = E(C_p) \cap \delta(M^p_j)$ and $|E(C_i) \cap \delta(M^p_j)| = 2$ hold for each triple $(i, p, j)$ which satisfies $1 \leq p \leq i$ and $1 \leq j \leq k_p$.}
\end{enumerate}
}
\end{claim}

We begin the proof of this claim with statement~(a). Here we use that the inclusions ${V(K^i_0) \cup \mathscr{S}^i \cup N_3(\mathscr{S}^i) \subseteq V(C_{i}) \subseteq V(K^{i+1}_0)}$ hold for every $i \geq 1$  by construction and Lemma~\ref{Ob_Su-cut-1}~(i). Since $N(K^{i}_0) = \mathscr{S}^i$ is true by definition and Lemma~\ref{struct_toll}, statement~(a) follows.

We fix an arbitrary integer $i \geq 1$ and some $j$ which satisfies $1 \leq j \leq k_i$ to prove statement~(b). By definition and Lemma~\ref{Ob_Su-cut-1}~(ii), we know that $M^i_j$ contains either all or no vertices of $K^i_p$ for every $p$ with $1 \leq p \leq k_i$. Using Lemma~\ref{struct_toll}, we obtain that $N(M^i_j) \subseteq V(K^i_0) \cup \mathscr{S}^i \cup N(\mathscr{S}^i)$. Since $K^i_0$ and $\mathscr{S}^i$ are finite by definition and Lemma~\ref{struct_toll} and $G$ is locally finite by assumption, we obtain that $\delta(M^i_j)$ is finite.

Let us fix an arbitrary end $\omega$ of $G$ for statement~(c). We use that for every $i \geq 1$ the end $\omega$ lies in precisely one of the closures $\overline{K^i_{1}}, \ldots, \overline{K^i_{k_i}}$, say $\omega \in \overline{K^i_{j}}$ where $1 \leq j \leq k_i$, since $K^i_0$ and $\mathscr{S}^i$ are finite by definition and Lemma~\ref{struct_toll}. Then we set $f(i) = j$. Now we show that ${M^{j}_{f(j)} \subseteq M^i_{f(i)}}$ holds for all integers $i, j$ with $1 \leq i \leq j$. Note that it suffices to prove the inclusion ${M^{i+1}_{f(i+1)} \subseteq M^i_{f(i)}}$ for every $i \geq 1$. The definition of $M^{i}_{f(i)}$ and Lemma~\ref{Ob_Su-cut-1}~(ii) ensure that $G[M^{i}_{f(i)}]$ is connected and that the inclusions ${V(K^{i}_{f(i)}) \subseteq M^{i}_{f(i)} \subseteq (V \setminus V(C_{i-1})) \cup N_2(N(C_{i-1}))}$ are valid for every $i \geq 1$. Note that the definition of $f$ implies the inclusion $V(K^{i+1}_{f(i+1)}) \subseteq V(K^{i}_{f(i)})$ for every $i \geq 1$. We need furthermore the observation that $M^{i+1}_{f(i+1)}$ does not contain a vertex of $\mathscr{S}^i$ for any $i \geq 1$. To see this, note that ${V(K^{i}_0) \cup \mathscr{S}^{i} \cup N_3(\mathscr{S}^{i}) \subseteq V(C_i)}$ holds for every $i \geq 1$ by definition of $C_i$ together with Lemma~\ref{Ob_Su-cut-1}~(i). So the distance between $\mathscr{S}^{i}$ and $\mathscr{S}^{i+1}$ is at least $4$ for every $i \geq 1$. Now the inclusion ${M^{i+1}_{f(i+1)} \subseteq (V \setminus V(C_{i})) \cup N_2(N(C_i))}$ implies that $M^{i+1}_{f(i+1)}$ cannot contain a vertex of $\mathscr{S}^i$ for any $i \geq 1$. Since we know for every $i \geq 1$ that $G[M^{i+1}_{f(i+1)}]$ is connected, $V(K^{i+1}_{f(i+1)})$ is a subset of $M^{i+1}_{f(i+1)}$ and $\mathscr{S}^{i}$ separates ${V \setminus (V(K^{i}_{f(i)}) \cup \mathscr{S}^{i})}$ from $V(K^{i}_{f(i)})$ by definition and Lemma~\ref{struct_toll}, we obtain that ${M^{i+1}_{f(i+1)} \subseteq V(K^{i}_{f(i)}) \subseteq M^{i}_{f(i)}}$ holds for every $i \geq 1$.

It remains to prove the equation ${M_{\omega} = \lbrace \omega \rbrace}$. As noted above, the inclusions ${V(K^i_{f(i)}) \subseteq M^i_{f(i)} \subseteq (V \setminus V(C_{i-1})) \cup N_2(N(C_{i-1}))}$ are true for every $i \geq 1$. So the definition of $f$ ensures that $\omega$ is an element of $M_{\omega}$. Next we show that $M_{\omega}$ does not contain a vertex of $G$ or any other end of $G$ than $\omega$. So let $v \in V(G)$ and $\omega' \neq \omega$ be an end of $G$. This means we can find a finite set of vertices $F \subseteq V(G)$ such that $\omega$ and $\omega'$ lie in closures of different components of ${G-F}$. Let $q \geq 1$ be an integer such that all vertices of $F \cup \lbrace v \rbrace$ are contained in $K^{q}_0$. We can find such an integer because each vertex $w \in F \cup \lbrace v \rbrace$ lies in some cycle $C_{\ell_w}$ with $\ell_w \geq 0$ by statement~(a). Using that $V(C_i) \subseteq V(C_{i+1})$ is true for every $i \geq 0$ by construction and Lemma~\ref{Ob_Su-cut-1}~(i) together with the fact that the set $F \cup \lbrace v \rbrace$ is finite, we can set $q-1$ as the maximum of all integers $\ell_w$. By definition of $K^{q}_0$ and Lemma~\ref{struct_toll}, we obtain that the inclusion $F \cup \lbrace v \rbrace \subseteq V(K^{q}_0)$ holds. This implies that $\omega$ and $\omega'$ lie also in closures of different components of $G-V(K^{q}_0)$. By construction and Lemma~\ref{Ob_Su-cut-1}~(ii), we know that the graph $G[M^{q+1}_{f(q+1)}]$ is connected. Furthermore, we get that the inclusions ${M^{q+1}_{f(q+1)} \subseteq (V \setminus V(C_{q})) \cup N_2(N(C_q))}$ and ${V(K^{q}_0) \cup \mathscr{S}^{q} \cup N_3(\mathscr{S}^{q}) \subseteq V(C_{q})}$ are true. Since the distance from $K^{q}_0$ to $N(C_q)$ is at least $5$ by construction, the set ${M^{q+1}_{f(q+1)}}$ cannot contain vertices of $K^{q}_0$. So $v$ is no element of $M^{q+1}_{f(q+1)}$. Now the connectedness of $G[M^{q+1}_{f(q+1)}]$ implies that $G[M^{q+1}_{f(q+1)}]$ is a subgraph of the component of $G-V(K^{q}_0)$ whose closure contains $\omega$. As $\omega$ and $\omega'$ lie in closures of different components of $G-V(K^{q}_0)$, the end $\omega'$ does not lie in the closure of $M^{q+1}_{f(q+1)}$ and we obtain that $v$ and $\omega'$ are no elements of $M_{\omega}$. Since each $M^{i}_{f(i)}$ is a vertex set, the intersection $M_{\omega}$ cannot contain inner points of edges. Therefore, the equation $M_{\omega} = \lbrace \omega \rbrace$ is valid, which shows statement~(c).

To prove statement~(d), take an edge $e \in E(C_i) \cap E(C_j)$ for arbitrary integers $i$ and $j$ that satisfy $0 \leq i < j$. So both endvertices of $e$ lie in $V(C_i) \subseteq V(K^{i+1}_0)$. Additionally, the inclusions ${V(K^{i+1}_0) \cup \mathscr{S}^{i+1} \cup N_3(\mathscr{S}^{i+1}) \subseteq V(C_{i+1}) \subseteq V(C_j)}$ are true by definition of the cycles and Lemma~\ref{Ob_Su-cut-1}~(i). Using the equality ${N(K^{i+1}_0) = \mathscr{S}^{i+1}}$, we conclude that ${e \in E(C_j - N_2(N(C_j)))}$ holds. So we get by definition of the cycles and Lemma~\ref{Ob_Su-cut-1}~(iii) that $e$ lies in $E(C_{j+1})$. This completes the proof of statement~(d).

Let us fix an arbitrary $p \geq 1$ and $j$ with $1 \leq j \leq k_p$ for statement~(e). We know that $|E(C_p) \cap \delta(M^p_j)| = 2$ holds by definition of the cycles and Lemma~\ref{Ob_Su-cut-1}~(ii). So it suffices to prove that $E(C_p) \cap \delta(M^p_j) = E(C_i) \cap \delta(M^p_j)$ holds for every $i \geq p$. Next let us consider an arbitrary edge $e = uv \in \delta(M^p_j)$. We prove now that $u$ and $v$ are contained in $V(K^p_0) \cup \mathscr{S}^p \cup N(\mathscr{S}^p)$ where at most one of these two vertices lies in $N(\mathscr{S}^p) \setminus V(K^p_0)$. If one of the endvertices of $e$ lies in $N(\mathscr{S}^p) \setminus V(K^p_0)$, it must be contained in $V(K^p_{j'}) \cap M^p_j$ for some $j'$ which satisfies $1 \leq j' \leq k_p$. Then the other endvertex of $e$ lies in $N(K^p_{j'}) \subseteq \mathscr{S}^p$ because the set $M^p_j$ must contain all vertices of $K^p_{j'}$ by definition and Lemma~\ref{Ob_Su-cut-1}~(ii). So the inclusion $\lbrace u,v \rbrace \subseteq \mathscr{S}^p \cup N(\mathscr{S}^p)$ is true. Otherwise, neither $u$ nor $v$ is a vertex of $N(\mathscr{S}^p) \setminus V(K^p_0)$. As precisely one endvertex of $e$ lies in $V \setminus M^p_j$ and $M^p_j$ contains either no or all vertices of $K^p_q$ for each $q$ with $1 \leq q \leq k_p$ by definition and Lemma~\ref{Ob_Su-cut-1}~(ii), we get that neither $u$ nor $v$ lies in any $K^p_q$ with $1 \leq q \leq k_p$. So both vertices must lie in $V(K^p_0) \cup \mathscr{S}^p$. Now we prove by induction on $i$ that  every edge $e' = u'v' \in E(C_p) \cap \delta(M^p_j)$ is an edge of $C_i$ for every $i \geq p$. For $i=p$, this is obvious. So assume that $e'$ is an edge of $C_i$ for some $i \geq p$. Note that ${V(K^{p}_0) \cup \mathscr{S}^p \cup N_3(\mathscr{S}^p) \subseteq V(C_i)}$ is true for every $i \geq p$ by definition and Lemma~\ref{Ob_Su-cut-1}~(i). Therefore, the edge $e'$ lies in $E(C_i - N_2(N(C_i)))$, which means that $e'$ is also an edge of $C_{i+1}$ by definition of the cycles and Lemma~\ref{Ob_Su-cut-1}~(iii). This completes the induction and shows that $e'$ is an edge of $C_i$ for every $i \geq p$. So the inclusion $E(C_p) \cap \delta(M^p_j) \subseteq E(C_i) \cap \delta(M^p_j)$ is true. We complete the proof of statement~(e) by showing by induction on $i$ that for every $i \geq p$ the cycle $C_i$ contains no edges of $\delta(M^p_j)$ but the two which are also edges of $C_p$. For $i=p$, this is obvious by definition of $C_p$ and Lemma~\ref{Ob_Su-cut-1}~(ii). So let $i > p$ and $e = uv \in \delta(M^p_j) \setminus E(C_p)$ be fixed for this purpose. Using the induction hypothesis, we know that $e \notin E(C_{i-1})$. Now suppose for a contradiction that $e$ is an edge of $C_i$. Then the definition of $C_i$ together with Lemma~\ref{Ob_Su-cut-1}~(iii) implies that $\lbrace u, v \rbrace \subseteq (V \setminus V(C_{i-1})) \cup N_3(N(C_{i-1}))$ holds. This leads towards a contradiction because we already know that the inclusion $\lbrace u,v \rbrace \subseteq V(K^p_0) \cup \mathscr{S}^p \cup N(\mathscr{S}^p)$ is valid where $\lbrace u, v \rbrace$ is no subset of $N(\mathscr{S}^p) \setminus V(K^p_0)$. That cannot be true since $C_{i-1}$ contains all vertices of $V(K^p_0) \cup \mathscr{S}^p \cup N_3(\mathscr{S}^{p})$ by definition of the cycle and Lemma~\ref{Ob_Su-cut-1}~(i). This completes the induction and we get that $E(C_p) \cap \delta(M^p_j) = E(C_i) \cap \delta(M^p_j)$ holds for every $i \geq p$. So the proof of statement~(e) is done and the claim is proved.
\newline

Using the sequence of cycles $(C_i)_{i \in \mathbb{N}}$, the integer sequence $(k_i)_{i \in \mathbb{N} \setminus \lbrace 0 \rbrace}$ and the vertex sets $M^i_j$ for every $i \in \mathbb{N} \setminus \lbrace 0 \rbrace$ and $j$ with $1 \leq j \leq k_i$, we can apply Lemma~\ref{HC-extract} because of Claim~1. So we get that $G$ is Hamiltonian.
\end{proof}

Next we state a couple of corollaries of Theorem~\ref{Inf_Ob_Su}. In order to do this, we need another definition. For an integer $k \geq 1$ and a graph $G$, the \textit{$k$-th power} $G^k$ of $G$ is the graph with vertex set $V(G)$ where two vertices are adjacent if and only if the distance between them in $G$ is at least $1$ and at most $k$. The analogues of the corollaries for finite graphs are all known and the proofs for the finite versions are the same as for the locally finite versions. The finite version of the following corollary is due to Matthews and Sumner \cite[Cor.\ 1]{MaSu}.

\begin{corollary}\label{Obsu-cor1}
Let $G$ be a locally finite connected graph with at least three vertices. If $G^2$ is claw-free, then $G^2$ is Hamiltonian.
\end{corollary}

The following three corollaries deal with line graphs. It is well known that this class of graphs forms a subclass of all claw-free graphs. The finite versions of Corollary \ref{Obsu-cor2} and Corollary \ref{Obsu-cor3} are due to Oberly and Sumner (see \cite[Cor.\ 1 and Cor.\ 3]{ObSu}).

\begin{corollary}\label{Obsu-cor2}
Let $G$ be a locally finite connected graph with at least three edges. If its line graph $L(G)$ is locally connected, then $L(G)$ is Hamiltonian.
\end{corollary}

The proof of the finite version of the next corollary in \cite[Cor.\ 3]{ObSu} shows that for a graph, the property of being locally connected is preserved under taking the line graph.

\begin{corollary}\label{Obsu-cor3}
For every locally finite, connected, locally connected graph with at least three vertices, its line graph is Hamiltonian.
\end{corollary}

The finite analogue of the next corollary has appeared in a paper of Nebesky (see \cite[Thm.\ 1]{Nebesky}). In \cite[Cor.\ 5]{ObSu} a proof of the finite version of the next corollary can be found using Theorem~\ref{Ob_Su}, the finite version of Theorem~\ref{Inf_Ob_Su}.

\begin{corollary}\label{Obsu-cor4}
Let $G$ be a locally finite connected graph with at least three vertices. Then $L(G^2)$ is Hamiltonian.
\end{corollary}

The finite version of the next corollary is due to Chartrand and Wall (see \cite{CharWall}). A proof of the finite result using Theorem~\ref{Ob_Su} can be found in \cite[Cor.\ 4]{ObSu}.

\begin{corollary}\label{Obsu-cor5}
Let $G$ be a locally finite connected graph with $\delta(G) \geq 3$. Then $L(L(G))$ is Hamiltonian.
\end{corollary}

The following corollary involves another class of graphs. We call a graph \textit{chordal} if it has no induced cycle with more than three vertices. Balakrishnan and \linebreak Paulraja \cite[Thm.\ 5]{BalaPaul} proved the finite analogue of the following corollary. They showed first that a graph which is $2$-connected and chordal has also the property of being locally connected. Then they applied Theorem~\ref{Ob_Su}.

\begin{corollary}\label{Obsu-cor6}
Every locally finite, $2$-connected, chordal, claw-free graph is \linebreak Hamiltonian.
\end{corollary}

\end{document}